\documentclass[12pt]{article}
\usepackage{amsmath,amssymb}
\usepackage{amsthm}
\usepackage{epsfig}
\usepackage{graphics}
\usepackage{graphicx}
\usepackage{colordvi}
\usepackage{mathrsfs} 
\usepackage{stmaryrd}
\usepackage{dsfont}
\usepackage{url}
\usepackage{hyperref}
\usepackage{xcolor,lipsum}

\hypersetup{colorlinks=true,urlcolor=blue, linkcolor=blue, citecolor=blue}
\usepackage{xr}
\oddsidemargin
0.2in \evensidemargin 0.2in\marginparwidth 0.4in

\newcommand{\1}{\mathds 1}

\newcommand{\G}{\mathscr G}

\newcommand{\vep}{{\varepsilon}}

\newcommand{\lra}{\longrightarrow}

\newcommand{\lmt}{\longmapsto}

\newcommand{\bs}{\boldsymbol}

\renewcommand{\P}{{\mathbb P}}
\newcommand{\E}{{\mathbb E}}

\newcommand{\tr}{{\sf tr}}
\newcommand{\R}{{\Bbb R}}

\newcommand{\ra}{\rangle}
\newcommand{\la}{\langle}
\newcommand{\diag}{{\sf diag}}

\newcommand{\vertiii}[1]{{\left\vert\kern-0.25ex\left\vert\kern-0.25ex\left\vert #1 
    \right\vert\kern-0.25ex\right\vert\kern-0.25ex\right\vert}}

\newtheoremstyle{slantthm}{10pt}{10pt}{\slshape}{}{\bfseries}{}{.5em}{\thmname{#1}\thmnumber{ #2}\thmnote{ (#3)}.}
\newtheoremstyle{slantrmk}{10pt}{10pt}{\rmfamily}{}{\bfseries}{}{.5em}{\thmname{#1}\thmnumber{ #2}\thmnote{ #3}.}

\begin{document}
\theoremstyle{slantthm}
\newtheorem{mthm}{Theorem}
\newtheorem{mcor}[mthm]{Corollary}
\newtheorem{thm}{Theorem}[section]
\newtheorem*{thmi}{Theorem}
\newtheorem{thmm}[thm]{Theorem}
\newtheorem{prop}[thm]{Proposition}
\newtheorem{lem}[thm]{Lemma}
\newtheorem{cor}[thm]{Corollary}
\newtheorem{defi}[thm]{Definition}
\newtheorem{prob}[thm]{Problem}
\newtheorem{disc}[thm]{Discussion}
\newtheorem{ko}[thm]{Key Observation/Assumption}
\newtheorem*{conj}{Conjecture}
\theoremstyle{slantrmk}
\newtheorem{ass}[thm]{Assumption}
\newtheorem{rrmk}[thm]{Research remark}
\newtheorem{rmk}[thm]{Remark}
\newtheorem*{exe}{Exercise}
\newtheorem{eg}[thm]{Example}
\newtheorem{quest}[thm]{Quest}
\newtheorem{comm}[thm]{Comment}
\newtheorem{nota}[thm]{Notation}
\newtheorem{rev}{Revision}

\numberwithin{rev}{section}
\numberwithin{equation}{section}
\allowdisplaybreaks[4]

\title{\bf Approximating Laplace transforms of meeting times for some symmetric Markov chains}
\author{Yu-Ting Chen\footnote{Center of Mathematical Sciences and Applications, Program for Evolutionary Dynamics, Harvard University
}}
\date{}

\maketitle

\abstract{We study distributions of meeting times for finite symmetric Markov chains. For Markov kernels defined on large state spaces which satisfy certain weak inhomogeneity in return probabilities of points up to large numbers of steps,
we obtain approximation, with explicit error bounds, of the Laplace transforms of some meeting times (without scaling) by ratios of Green functions closely related to hitting times of points.
In studying this approximation, we identify
some key matrix power series in Markov kernels weighted with solutions to a discrete transport-like equation with explicit coefficients, which stems from the viewpoint that meeting time distributions are equivalent to correlations of some linear particle system.
Our result applies in particular to random walks on large random regular graphs. It gives a justification  
of the corresponding practice, among other things, in Allen, Traulsen, Tarnita and Nowak \cite{ATTN} on approximating certain critical values for
the emergence of cooperation when mutation is present.}

\section{Introduction}\label{sec:intro}
In this paper we investigate 
distributions of (first) meeting times by two independent continuous-time Markov chains. 
Meeting times arise as fundamental objects in the study of finite stochastic spatial models in mathematical biology, including, for example, voter models \cite{C:CRW, CCC:WF}, Kimura's stepping-stone model and its generalizations \cite{K:SSM, S:SSM, CD:SSM,CDZ:SSM2}, and evolutionary games on graphs \cite{ATTN,C:BC}. In contrast to the frequently visited subject in most of these studies on approximating meeting time distributions under large time scales by exponential variables (cf. Keilson~\cite{K:AE} and Aldous and Fill~\cite[Section 3.5.4]{AF:RMC} for the classical results), our main focus in this paper is on the detailed meeting time distributions \emph{up to} large times through their Laplace transforms. This direction concerns in particular the biological application of
approximating certain critical values for the emergence of cooperation under some evolutionary games. For this, a brief discussion of the background and pointers to the literature will be given in Section~\ref{sec:EGT}.

Unless otherwise mentioned, we consider throughout this paper irreducible (conservative) Markov kernels $Q$ on finite state spaces $E$ with $\#E=N> 8$ 
and subject to the symmetry condition:
\begin{align}\label{ass:sym}
Q(x,y)=Q(y,x)\quad \forall\;x,y\in E
\end{align}
as well as the zero-trace condition:
\begin{align}\label{ass:trace}
\tr(Q)=0.
\end{align}
For such a Markov kernel $Q$, 
we denote by $M_{x,y}$ the (first) meeting time of two independent continuous-time rate-$1$ $Q$-chains started at $x$ and $y$, and hence, $M_{x,x}\equiv 0$. In addition, we write $H_{x,y}$ for the first hitting time of $y$ by a rate-$1$ $Q$-chain $X^x$ started at $x$.

One of the conventional treatments of meeting times is based on the validation of the following exact distributional reduction: 
\begin{align}\label{eq:MT}
M_{x,y}\stackrel{(\rm d)}{=}H_{x,y}/2.
\end{align}
Observe that (\ref{eq:MT}) means a significant dimensionality reduction in invoking both the number of Markov chains and the size of sets to be hit, and so the driving Markov kernel can be made stand out through the associated Green functions for 
a detailed study of meeting time distributions by
\begin{align}\label{eq:HTGF}
\begin{split}
&\E\big[e^{-\lambda H_{x,y}/2}\big]\\
&\hspace{.5cm}=\left.\E\left[\int_0^\infty e^{-\lambda t/2}\1_{\{y\}}(X^x_t)dt\right]\right/\E\left[\int_0^\infty e^{-\lambda t/2}\1_{\{y\}}(X^y_t)dt\right].
\end{split}
\end{align}
Although (\ref{eq:MT}) can lead to 
this important relation and is immediate for some special Markov kernels (the classical examples are random walks on groups), obtaining the reduction (\ref{eq:MT}) in general poses some problems. For example, while we restrict our attention to Markov kernels subject to the symmetry condition (\ref{ass:sym}) in this paper, it is readily pointed out in Aldous~\cite[p.188]{A:MTMC} by the counterexample of random walks on star graphs that (\ref{eq:MT}) can fail when the symmetry condition is disobeyed. In addition, it seems that before the present paper, the reduction (\ref{eq:MT}) can only be obtained by few arguments which require
various strong links among all pointwise transition probabilities, and thus, are insufficient for a further study of
whether some ``non-global'' properties of Markov chains similar to those allowing for the exact equality (\ref{eq:MT}) can give an approximating reduction.

Our goal in this paper is to determine whether the distributional reduction (\ref{eq:MT}) is approximately correct in terms of Laplace transforms, 
for a class of ``weakly inhomogeneous'' Markov kernels. (See also Section~\ref{sec:EGT} for more on our motivation.) Informally, we consider kernels $Q$ on large, but finite, state spaces $E$ under which pointwise return probabilities are almost identical up to some large numbers of steps, and the presence of few ``wild'' points which disagree significantly with the majority on such almost homogeneity of finite-range return probabilities is allowed (see our discussion after Theorem~\ref{thmm:main-1} for the example of large random regular graphs). To formalize this, we choose
the largest sets of points which almost agree in $s$-step return probabilities under $Q$, for $s\in \Bbb Z_+$, from  
\begin{align}\label{def:RQ}
\mathcal R_Q^\gamma(x,s)\triangleq \{y\in E;|Q^s(x,x)-Q^s(y,y)|\leq \gamma\},\quad x\in E,\;\gamma\in [0,1],
\end{align}
and consider, for every $s_0\in \Bbb N$, the magnitude of
\begin{align}\label{context}
\frac{\min_{x\in E}[N-\# \mathcal R^{\gamma}_{Q}(x,s)]}{N}+\gamma
\end{align}
for all $s\in \{0,\cdots, s_0\}$. We stress that \emph{no} connections between such kernels and those which allow for the conventional distributional reduction of meeting times will be assumed a priori, and the novelty of this paper as we see it lies in the treatment of this missing link in order to understand how
 Laplace transforms of some meeting times and those of the associated hitting times of points (as in (\ref{eq:MT})) can still be related. See Section~\ref{sec:method} for further discussions.

To facilitate the present purpose to study meeting times on large sets, let us introduce some meeting times started from pairs of random points independent of the underlying $Q$-Markov chains. 
The first natural candidate 
is the pair of random points $(U,V_\infty)$ whose coordinates are independent and uniformly distributed on the state space $E$. 
See, for example, Durrett~\cite{D:Epi} for a discussion of the corresponding meeting time $M_{U,V_\infty}$, and Chen, Choi and Cox~\cite{CCC:WF} for its role in diffusion approximation of voter models defined on large sets. On the other hand, if we consider $V_s$ for $s\in \Bbb N$ being distributed according to the $s$-step distribution $Q^s(U,\,\cdot\,)$ of the discrete-time $Q$-chain started at $U$ when we condition on $U$, then the meeting times $M_{U,V_s}$ have a very different nature. For example, if we consider random walk on a large graph with relatively smaller degree, then informally, $M_{U,V_s}$ for a small $s\in \Bbb N$ is the first meeting time of two Markov chains started with points close to each other (with respect to graph distance) such that the local structure of the underlying graph becomes an essential part of its distribution, whereas $M_{U,V_\infty}$ is for starting points far away from each other and should depend much less on the local structure of the graph (see Oliveira~\cite{O:MFC} and Chen et al.~\cite[Theorem~4.1]{CCC:WF} for some rigorous treatments).

For $s\in \Bbb N\cup \{\infty\}$, we will call $M_{U,V_s}$ the {\bf $\bs s$-th order meeting time}, but 
put emphasis on the special case $s=1$ throughout this paper. There are several reasons why we choose to make this restriction. First,
 the distributions of $M_{U,V_1}$ and $M_{U,V_\infty}$ always determine each other through an explicit integral equation  
(cf. Aldous and Fill~\cite[Proposition 3.21]{AF:RMC} or Section 3 of Chen et al.~\cite{CCC:WF}). 
More importantly, since the mass of $M_{U,V_\infty}$ always escapes to infinity in the limit of a large state space (see \cite[Corollary~3.4]{CCC:WF}) and our interest here is for meeting times under the natural time scale, the viable object of study need be the distribution of $M_{U,V_1}$. Besides, the tail distribution of $M_{U,V_{s_0}}$ for finite $s_0$ can be approximated by that of $M_{U,V_1}$ if it holds that (\ref{context}) is small for $s $ up to $s_0$ (Section~\ref{sec:MT}). 
In accordance with this emphasis on the first-order meeting time and to save notation, we write $V$ for $V_1$ from now on.

The main result of this paper is given by Theorem~\ref{thmm:main-1} below. For the statement of the theorem, we define
\begin{align}\label{def:deltaQ}
\Delta_Q^{\gamma}\triangleq \min_{s\in \Bbb N}\left\{\frac{\min_{x\in E}[N-\#\mathcal R^\gamma_Q(x,s)]}{N}+\gamma,\;\;\frac{\tr(|Q|^s)}{N}\right\}
\end{align}
(recall (\ref{context})),
where $f(Q)$ is defined by functional calculus of $Q$ for every complex function $f$ defined on an open set in $\Bbb C$ containing the line segment $ [-1,1]$ and $\tr(\,\cdot\,)$ denotes trace, so that $\tr\big(f(Q)\big)=\sum f(q)$ for $q$ ranging over all eigenvalues of $Q$.

\begin{thmm}\label{thmm:main-1}
Let $Q$ be an irreducible finite Markov kernel
subject to the symmetry condition (\ref{ass:sym}) and the zero-trace condition (\ref{ass:trace}). Suppose that $Q$ is defined on a finite set $E$ of size $N=\#E> 8$. Let $\vep\in (0,1]$ and $\lambda\in (\vep,\infty)$ such that $(\lambda-\vep)N-\lambda\vep>0$, and let 
$m\in \Bbb N$ and $\gamma\in [0,1]$ 
be auxiliary parameters left to be chosen.
We have
\begin{align}\label{crit:main0}
&\left|\E[e^{-\lambda M_{U,V}}]-\left.\E\left[\int_0^\infty e^{-t\lambda /2}\1_{\{V\}}(X_t^{U})dt\right]\right/\E\left[\int_0^\infty e^{-t\lambda /2}\1_{\{V\}}(X_t^{V})dt\right]\right|\notag\\
\begin{split}
\leq &4\Big(1+\frac{\lambda}{N}\Big)^{-mN}+\left[\Big(1+\frac{\lambda}{N}\Big)\Big(1-\frac{\vep}{N}\Big)\right]^{-(mN+1)}\times \frac{C_\vep\lambda(N-\vep)}{(\lambda-\vep)N-\lambda\vep}\\
&\hspace{.5cm}+80\Delta^\gamma_Q\times \left(1+\frac{6}{N}\right)^{mN},
\end{split}
\end{align}
where $C_\vep$ is a constant which diverges as $\vep\searrow 0$ and can be chosen to be
\begin{align}
C_\vep &=\frac{1}{\pi}\left(2\vep^{-1}+9+\pi+\frac{32\sqrt{2}}{(1-\cos 1)^{1/2}}\right)\times \frac{4+\vep}{\vep},\label{def:Cvep}
\end{align}
and $\Delta_Q^{\gamma}$ is defined as in (\ref{def:deltaQ}).
\end{thmm}

We will discuss the method of proof for Theorem~\ref{thmm:main-1} in Section~\ref{sec:method}, and start its proof in Section~\ref{sec:tpc}, with the conclusion set in Section \ref{sec:lap-2}.

Observe that the ratio of Green functions on the left-hand side of (\ref{crit:main0}) bears a strong resemblance to the Laplace transform $\E[e^{-\lambda H_{U,V}/2}]$ if we recall (\ref{eq:HTGF}). For this reason, we regard (\ref{crit:main0}) as an approximation in a weak sense for the classical distributional reduction (\ref{eq:MT}). In addition, this ratio can be written as
\begin{align}
\begin{split}\label{ratio}
&\left.\E\left[\int_0^\infty e^{-t\lambda /2}\1_{\{V\}}(X_t^{U})dt\right]\right/\E\left[\int_0^\infty e^{-t\lambda /2}\1_{\{V\}}(X_t^{V})dt\right]\\
&\hspace{3.5cm}=\left.\tr\left(\frac{Q}{\lambda+2(1-Q)}\right)\right/\tr\left(\frac{1}{\lambda+2(1-Q)}\right),
\end{split}
\end{align}
which depends only on the eigenvalue distribution of $Q$ (see the proof of Lemma~\ref{lem:MUV-exp} (iii)).

Let us discuss informally some contexts 
for which Theorem~\ref{thmm:main-1} may allow for ``good'' approximation of the Laplace transform of $M_{U,V}$ by the associated ratio of Green functions.
The key, of course, is to identify kernels $Q$ such that the associated parameters $\Delta^\gamma_Q$ in (\ref{def:deltaQ}) are small. 
We consider a kernel $Q$ such that the quantity in (\ref{context}) is small for $s$ up to some $s_0$, so the task of bounding $\Delta^\gamma_Q$ falls upon controlling $\tr(|Q|^{s_0})/N$. 
In this direction, if the eigenvalues of $Q$ are sufficiently bounded away from $\pm 1$ except for a small fraction of them, then the normalized trace term $\tr(|Q|^{s_0})/N$ is small if the magnitude of $s_0$ can be chosen relatively larger. (See also the role of spectral gaps of Markov chains in obtaining almost exponentiality of hitting times in Aldous and Fill~\cite[Section 3.5.4]{AF:RMC}.) 

In fact, if $Q$ is the random walk kernel on a (simple connected) $k$-regular graph on $N$ vertices with $k\geq 2$ and $s_0$ is even, we can alternatively bound $\tr(|Q|^{s_0})/N$ by means of partial geometry of the graph and the analogous spectral functional on the infinite $k$-regular tree.
(Here and in what follows, see Biggs \cite{B:AGT} for terminology in algebraic graph theory.) Indeed, we have
\[
\frac{\tr(|Q|^{s_0})}{N}\leq\frac{N-N(s_0)}{N}+ \int_{-1}^1 q^{s_0}f^{(k)}(q)dq,
\] 
where $N(s_0)$ is the number of vertices $x$ in the regular graph such that the subgraph induced by the vertices at graph distance $\leq s_0/2$ from $x$ defines a tree (cf. McKay~\cite[Lemma~2.2]{M:EED}), and 
\begin{align}\label{KM:SM}
f^{(k)}(q)=\frac{\sqrt{4(k-1)-(kq)^2}}{2\pi (1-q^2)},\quad q\in \left[-\frac{2\sqrt{k-1}}{k},\frac{2\sqrt{k-1}}{k}\,\right],
\end{align}
is the Kesten-McKay
density \cite{K:SRG, M:EED} for the spectral measure of the random walk kernel on the infinite $k$-regular tree and satisfies
$\lim_{s\to\infty}\int q^{s}f^{(k)}(q)dq=0$.
For approximation of the ratio of Green functions as in (\ref{crit:main0}) by explicit values using the spectral representation (\ref{ratio}) on regular graphs with large girth, see McKay~\cite[Theorem 4.4]{M:EED}.

\begin{cor}
\label{thmm:main-1-0}
Let $\{Q^{(n)}\}$ be a sequence of irreducible  finite Markov kernels subject to the symmetry condition (\ref{ass:sym}) and the zero-trace condition (\ref{ass:trace}). Assume that these kernels are defined on growing state spaces. If 
\begin{align}\label{crit:main}
\inf_{\gamma\in [0,1]}\limsup_{n\to\infty}
\Delta_{Q^{(n)}}^{\gamma}=0,
\end{align}
then
\begin{align}
\begin{split}\label{conv:main}
&\lim_{n\to\infty}\Bigg\{\E^{(n)}[e^{-\lambda M_{U,V}}]-\left.\E^{(n)}\left[\int_0^\infty e^{-\lambda t/2}\1_{\{V\}}(X^U_t)dt\right]\right/\\
&\hspace{2.5cm}\E^{(n)}\left[\int_0^\infty e^{-\lambda t/2}\1_{\{V\}}(X^{V}_t)dt\right]\Bigg\}=0,\quad \forall\;\lambda\in (0,\infty),
\end{split}
\end{align}
where $\E^{(n)}$ denotes expectation with respect to $Q^{(n)}$.
\end{cor}
\begin{proof}
Fix $(\vep,\lambda)$ such that $\vep\in (0,1]$ and $0<\vep<\lambda<\infty$.
Under (\ref{crit:main}), we pass limit for the analogues of (\ref{crit:main0}) for $Q^{(n)}$ first along the state space size. For the resulting error bounds, we take infimum over $\gamma\in [0,1]$ and then pass limit along $m\to\infty$.
\end{proof}

By Corollary~\ref{thmm:main-1-0}, we obtain the following particular result for random regular graphs (see, e.g., Wormald~\cite{W:MRG} for a survey of random regular graphs). A further discussion of this result will be given in Section~\ref{sec:EGT}.

\begin{cor}[Random regular graphs]\label{cor:main2}
Fix an integer $k\geq 2$.
With respect to the sequence of random walk kernels $\{Q^{(n)}\}$ associated with an i.i.d. sequence of growing (uniform) random $k$-regular graphs, it holds that
\begin{align}\label{eq:RGconv}
\lim_{n\to\infty}\E^{(n)}[e^{-\lambda M_{U,V}}]=\E^{(\infty)}[e^{-\lambda H_{x,y}/2};H_{x,y}<\infty]\quad \forall\;\lambda\in (0,\infty)
\end{align}
almost surely with respect to the randomness that the graphs are chosen.
Here $H_{x,y}$ under $\E^{(\infty)}$ denotes the first hitting time of $y$ by a rate-$1$ random walk on the (transient) infinite $k$-regular tree started at $x$ for any vertices $x$ and $y$ which are adjacent to each other.
\end{cor}

The proof of Corollary~\ref{cor:main2} is relegated to Section~\ref{sec:cor-main2}.

We recall that with probability tending to $1$ in the limit of a large graph size, random regular graphs with a fixed degree do not have homogeneity in return probabilities in all numbers of steps (cf. Wormald~\cite[Section 2.3]{W:MRG}). Nonetheless (\ref{crit:main}) holds almost surely with respect to the randomness that the graphs are chosen, as will be explained in Section~\ref{sec:cor-main2}.

\subsection{Application to evolutionary game theory}\label{sec:EGT}
The present study was in part motivated by Allen, Traulsen, Tarnita and Nowak \cite{ATTN}, and let us discuss briefly the role of the meeting times under consideration in their studies. The paper \cite{ATTN} investigates the influence of mutation on forming cliques of strategy types under some evolutionary games. The underlying game players are arranged according to finite vertex-transitive graphs, and their strategy types are updated indefinitely in some Markovian manners. Analogues of the key quantities on some infinite vertex-transitive graphs are also considered in \cite{ATTN}. See Nowak~\cite{N:ED} for an authoritative introduction to evolutionary games, Section 1 and 2 in the supplementary information of Nowak, Tarnita and Wilson \cite{NTW:EU} for an analysis of critical values under general
evolutionary games, and the recent survey paper by Allen and Nowak \cite{AN:GG} for critical values under games closely related to those in \cite{ATTN} on weighted vertex-transitive graphs.

Allen et al.~\cite{ATTN} obtains explicit results on certain critical values for the emergence of cooperation between game players, when special finite vertex-transitive graphs are taken into consideration. 
This is based on the result that on general finite vertex-transitive graphs,
the so-called identity-by-descent probabilities for two game players occupying adjacent vertices \cite[Section 2.3]{ATTN} are the only quantities left to be solved for the critical values \cite[Appendix C]{ATTN}. 
These identity-by-descent probabilities are equivalent to the Laplace transforms $\E[e^{-\lambda M_{x,y}}]$ of the first meeting times $M_{x,y}$ by independent continuous-time (rate-$1$) random walks started at adjacent vertices $x$ and $y$ (cf. the algebraic equations in \cite[Appendix B]{ATTN}), 
where $\lambda$ are
strictly positive parameters depending only on the underlying mutation rates. More generally, it is not difficult to check that on regular graphs, the Laplace transforms $\E[e^{-\lambda M_{U,V}}]$ determine the critical values for cooperation under the games in \cite{ATTN} 
(see \cite[Appendix A to C]{ATTN} and Corollary~\ref{cor:Mtime1}).

While the same evolutionary games
considered in Allen et al.~\cite{ATTN} in the absence of mutation allow for rather complete explicit results (see e.g. \cite{OHLN, TDW, CDP:VMP, C:BC}),
we are unable to expect so for the critical values in general
when mutation enters because Laplace transforms of meeting times are now involved. Nonetheless, as one special case, Allen et al. \cite[Section 2.2 and 3.2]{ATTN} circumvented the difficulty of handling the identity-by-descent probabilities equivalent to $\E[e^{-\lambda M_{U,V}}]$
on general finite regular graphs, by considering large random regular graphs and turning to the analogous identity-by-descent probabilities on infinite regular trees instead (or equivalently, $\E^{(\infty)}[e^{-\lambda M_{x,y}}]$ by the notation of Corollary~\ref{cor:main2}).
See also Szab\'o and F\'ath \cite{SF:EGG} for related discussions on the application of random regular graphs for evolutionary games on graphs. 

Corollary~\ref{thmm:main-1-0} gives a justification of the practice in Allen et al.~\cite{ATTN} discussed above if we assume a fixed mutation rate and take limit of the identity-by-descent probabilities on growing random regular graphs with a constant degree. In fact, by formalizing the discussion on approximating the Laplace transforms of first-order meeting times after Theorem~\ref{thmm:main-1}, a quantitative approximation of the critical values on finite regular graphs with bounded degree and large girth without passing to the limit is now also possible.

\subsection{Method of proof}\label{sec:method}
We consider a linear analysis for first meeting time distributions. 
Our point of view starts with the duality between
coalescing Markov chains and voter models, and we apply the linear coupling of voter models defined by products of i.i.d. random linear operators acting on initial configurations, which makes some computations natural and straightforward. This coupling of voter models and a study of similarly defined interacting particle systems can be found in Liggett \cite[Chapter IX]{L:IPS} (see e.g. \cite{FF:RAP, AL:AP, FY:LS} for recent studies).

Let us recall coalescing Markov chains, voter models and a particular result of their duality.
A system of coalescing Markov chains consists of continuous-time rate-$1$ $Q$-Markov chains started at all points of the state space $E$. They move independently before meeting  and together afterward. The associated voter model is a continuous-time rate-$N$ Markov chain $(\xi_t)$ taking values in the space of configurations of two possible ``opinions'', say $1$ and $0$, at points of the state space so that, at each epoch time, the opinion of a randomly chosen point, say $U'$, is changed to the opinion of another chosen according to $Q(U',\,\cdot\,)$.  
We consider the following particular consequence of the duality between the voter model and the coalescing Markov chains:
\begin{align}\label{eq:MTVM}
&\P(M_{x,y}>t)=\frac{1}{u(1-u)}\Big(u-\E_{\beta_u}\left[\xi_t(x)\xi_t(y)\right]\Big),\quad u\in (0,1),
\end{align}
for $x,y\in E$.
Here, under $\E_{\xi}$, the voter model starts at configuration $\xi\in \{1,0\}^E$, and $\E_{\beta_u}$ means that the initial configuration is randomized according to the random configuration $\beta_u$ which is defined by placing i.i.d. Bernoulli random variables $\beta_u(x)$ with mean $u\in (0,1)$ at all points $x$. See Liggett \cite{L:IPS}
for a general account of voter models and coalescing Markov chains.

Let us discuss 
how the duality method is applied for the study of meeting time distributions. It should be clear from the above description that
the voter model can be identified as a Markov chain on configurations which is updated sequentially by a family of i.i.d. random linear operators $(T_n)$ independent of the initial configuration. We may assume that the voter model is obtained by time-changing a Markov sequence $(\xi_n)$ by an independent rate-$N$ Poisson process, where $(\xi_n)$ is defined recursively by
\begin{align}\label{coupling}
\xi_n=T_n\xi_{n-1},\quad n\in \Bbb N.
\end{align}
Here, configurations are regarded as column vectors with coordinates indexed by points of $E$ (with respect to a fixed order).
(A detailed description of the distribution of these random linear operators is given in the proof of Proposition~\ref{prop:conj}.)
Write $J$ for the square matrix with entries identically equal to $1/N$, and let us use the standard bra-ket notation (see Paratharasy~\cite{P:QSC}) so that $\langle \beta_u|$ and $|\beta_u\rangle$ mean the row vector and the column vector corresponding to the Bernoulli random configuration $\beta_u$.
Then by the independence of the underlying objects defining the voter model started at $\beta_u$, the discrete-time analogues of the two-point correlations 
$\E_{\beta_u}[\xi_t(U)\xi_t(V_\infty)]$, 
which determine the distribution of the $\infty$-order meeting time $M_{U,V_\infty}$ through the duality equation (\ref{eq:MTVM}),
can
be written as:
\begin{align}\label{eq:coupling}
\E_{\beta_u}[\xi_n(U)\xi_n(V_\infty)]=&
\frac{1}{N}\E[\langle \beta_u|T_1^*\cdots T_n^*JT_n\cdots T_1|\beta_u\rangle]\notag\notag\\
=&\frac{1}{N}\E[\langle \beta_u|L^{n}(J)|\beta_u\rangle],\quad n\in \Bbb N,
\end{align}
where $L$ is a (deterministic) linear operator defined by the expected matrix ``congruence''
\begin{align}\label{eq:T*T}
L(C)\triangleq \E[T^*CT]
\end{align}
on the space of square matrices indexed by points of the state space. Here in (\ref{eq:T*T}), $T$ has the same distribution as a random linear operator $T_n$ 
in the linear coupling (\ref{coupling}) of the voter model, and $*$ denotes transpose. See also Liggett \cite[Section IX.3]{L:IPS} for characterizations of two-point correlations for general linear interacting particle systems by linear evolution
equations.

It is true that in general, the key matrices $L^n(J)$ in (\ref{eq:coupling}) are complicated, and hence, so are the associated discrete-time two-point correlations there. When, however, the return probabilities $Q^s(x,x)$ of the discrete-time $Q$-chain do not depend on points $x$ for every $s\in \Bbb Z_+$ or, simply, $Q$ is {\bf walk-regular} (see Godsil and McKay \cite{GM:WRG}), 
we find a detailed description for $L^n(J)$:
\begin{align}\label{eq:Ln-intro}
L^n(J)=J+\sum_{s=0}^\infty \alpha(n,s)Q^s.
\end{align}
Here, $\alpha$ is a two-parameter discrete function defined by an explicit partial recurrence equation, which in part takes the form of a discrete transport equation, has a finite-speed of propagation and is defined explicitly in terms of the traces $\tr(Q^s)$ of products of $Q$ (see Lemma~\ref{lem:alpha} for its definition). In particular, the resulting characterization of the discrete-time two-point correlations $\E_{\beta_u}[\xi_n(U)\xi_n(V_\infty)]$ is enough for an explicit expression of the Laplace transform of the \emph{first-order} meeting time, which is essentially due to the connection mentioned above between the distributions of $M_{U,V}$ and $M_{U,V_\infty}$ (see Section~\ref{sec:lap-1} for the details). Then it holds that
\begin{align}\label{eq:HT}
\begin{split}
&\E[e^{-\lambda M_{U,V}}]=\\
&\hspace{.5cm}=\left.\E\left[\int_0^\infty e^{-t\lambda /2}\1_{\{V\}}(X_t^{U})dt\right]\right/\E\left[\int_0^\infty e^{-t\lambda /2}\1_{\{V\}}(X_t^{V})dt\right]
\end{split}
\end{align}
(Corollary~\ref{cor:MUV-wr}), and so a comparison with the analogous representation (\ref{eq:HTGF}) for $\E[e^{-\lambda H_{U,V}/2}]$ entails the classical reduction 
\[
M_{U,V}\stackrel{(\rm d)}{=}H_{U,V}/2
\]
\emph{under} the assumption of walk-regularity of the underlying Markov kernel. We remark that the simple first-step recurrence argument is in fact enough to yield the equality (\ref{eq:HT}) (cf. Lemma 12 in \cite{AN:GG}). Nonetheless, it appears less clear why the explicit forms of the discrete-time two-point correlations  $\E_{\beta_u}[\xi_n(U)\xi_n(V_\infty)]$ of the voter model, implied by the infinite series (\ref{eq:Ln-intro}) for $L^n(J)$ and (\ref{eq:coupling}), can be as well obtained by inverting explicit forms for Laplace transforms of meeting times.

To establish an approximating equality for (\ref{eq:HT}) when walk-regularity of $Q$ fails, we introduce the terms
\begin{align}\label{eq:pc}
\frac{1}{N}\E[\langle \beta_u|L_0^n(J)|\beta_u\rangle]
\end{align}
as substitutes of the true correlations in (\ref{eq:coupling}), where the linear operator $L_0$ is chosen to be an ``approximating'' version of $L$  and such that
the infinite series expression in (\ref{eq:Ln-intro}) remains valid with $L^n(J)$ replaced by the new products $L^n_0(J)$. (See (\ref{def:L0}) for our choice of $L_0$.) For convenience, we will call the above terms {\bf approximating correlations} in view of (\ref{eq:coupling}), even though in general their probabilistic interpretation  is not clear to us. The main technical issues as a result of adopting these approximation correlations arise essentially in the context of large state spaces. We have to understand the asymptotic behavior of $L_0^n(J)$ for large $n$ with respect to the size of the state space, and bound the errors from replacing $L$ with $L_0$ in formulating the approximating correlations so that they do not grow too fast up to moderately large times even on large state spaces. The details are given in Section~\ref{sec:lap-2}.

\paragraph{\bf Organization of the paper.}
In Section~\ref{sec:tpc}, we study basic properties of the operator $L$ defined by (\ref{eq:T*T}), which gives the two-point correlations (\ref{eq:coupling}) of voter models, from its explicit expression to some related operator norms. Based on the explicit form, we define in Section~\ref{sec:ptpc} the operator $L_0$ in making the approximating correlations (\ref{eq:pc}), and study generating functions of the matrices $L_0^n(J)$. We show a preliminary connection among the distribution of $M_{U,V}$, the true correlations  (\ref{eq:coupling}), and the approximating correlations (\ref{eq:pc}) in Section~\ref{sec:lap-1}, and then in Section~\ref{sec:lap-2}, we prove the main approximation theorem (Theorem~\ref{thmm:main-1}) on the distributional reduction of the first-order meeting times. Details of the proof of Corollary~\ref{cor:main2} are given afterward in Section~\ref{sec:cor-main2}.
We close this paper with some relations between meeting times of all finite orders in Section~\ref{sec:MT}.

\section{Correlations of voter models}\label{sec:tpc}
Recall the linear operator $L$ defined in (\ref{eq:T*T}) by a matrix expectation. In this section, we derive its explicit representation and study related operator norms. From now on, we index entries of $N\times N$ matrices over $\Bbb C$ by points of $E$ (with a fixed order as before), and $\mathsf M_E$ stands for the linear space of $N\times N$ matrices subject to such convention. In addition, we will consistently use the bra-ket notation for matrix and vector multiplications.

To start with, let us set some matrix notation for the explicit representation of $L$. We define a linear operator $\diag$ on $\mathsf M_E$ by
\begin{align}
\diag(C)\triangleq \sum_{x\in E}|x\rangle\langle x|C|x\rangle \langle x|,\label{def:diag}
\end{align}
that is $\diag(C)$ is the diagonal matrix obtained from $C\in \mathsf M_E$ by setting all off-diagonal entries to zero and leaving all diagonal terms unchanged. Here in (\ref{def:diag}), $|x\ra$ and $\la x|$ denote respectively the column vector and the row vector with $1$ at $x$-th coordinate and zero otherwise.
 In addition, we write $J$ for the probability matrix defined by 
\begin{align}
\forall\;x,y\in E,\quad \langle x|J|y\rangle\triangleq \frac{1}{N}.\label{def:J}
\end{align}

\begin{prop}\label{prop:conj}
Fix a Markov kernel $Q$, and recall the associated linear operator $L$ which is defined by the matrix expectation (\ref{eq:T*T}). Then $L$ admits the explicit expression:
\begin{align}
\begin{split}
L(C)=&\frac{N-2}{N}C+\frac{1}{N}(CQ+QC)-\frac{1}{N}[\diag(C)Q+Q\diag(C)]\\
&+\frac{1}{N}\diag(C )
+\diag\big(Q\diag (C )J\big).\label{eq:tst}
\end{split}
\end{align}
\end{prop}
\begin{proof}
Fix a constant matrix $C\in \mathsf M_E$ throughout this proof.

For the present purpose to compute the explicit expression of $L$, we may assume that the operator $T$ in its definition (\ref{eq:T*T}) is given by
\begin{align*}
T=I-|U\ra \la U|+|U\ra\la V|.
\end{align*}
Here, as before, $\P(U=x,V=y)\equiv \frac{1}{N}\langle x|Q|y\rangle$.
Then we expand the product $T^*CT$ defining $L$ as
\begin{align}
\begin{split}\label{eq:T*CT}
T^*CT=&\big(I-|U\rangle\langle U|+|V\rangle \langle U|\big)C\big(I-|U\rangle\langle U|+|U\rangle \langle V|\big)\\
=&C-C|U\rangle \langle U|+C|U\rangle \langle V|\\
&-|U\rangle \langle U|C+|U\rangle \langle U|C|U\rangle \langle U|-|U\rangle \langle U|C|U\rangle \langle V|\\
&+|V\rangle \langle U|C-|V\rangle \langle U|C|U\rangle \langle U|+|V\rangle \langle U|C|U\rangle \langle V|.
\end{split}
\end{align}
Taking expectation for both sides of the foregoing equality, we obtain
\begin{align*}
L(C)=&\,\E\left[T^*CT\right]
=C-\sum_{x\in E}\frac{1}{N}C|x\rangle \langle x|+\sum_{x,y\in E}\frac{1}{N}\langle x|Q|y\rangle C |x\rangle \langle y|\\
&-\sum_{x\in E}\frac{1}{N}|x\rangle \langle x| C+\sum_{x\in E}\frac{1}{N}|x\rangle \langle x|C|x\rangle \langle x|-\sum_{x,y\in E}\frac{1}{N}\langle x|Q|y\rangle|x\rangle \langle x|C|x\rangle \langle y|\\
&+\sum_{x,y\in E}\frac{1}{N}\langle x|Q|y\rangle |y\rangle \langle x|C-\sum_{x,y\in E}\frac{1}{N}\langle x|Q|y\rangle|y\rangle  \langle x|C|x\rangle\langle x|\\
&+\sum_{x,y\in E}\frac{1}{N}\langle x|Q|y\rangle|y\rangle \langle x|C|x\rangle\langle y|.
\end{align*}
To simplify the right-hand side of the foregoing equality, we use the assumed symmetry of the Markov kernel $Q$ and the equality $I=\sum_{x\in E}|x\rangle \langle x|$, and invoke the operators $\diag(\,\cdot\,)$ and $J$ defined above in (\ref{def:diag}) and (\ref{def:J}), respectively. It follows that 
\begin{align}
\begin{split}\label{eq:tst1}
L(C)=&C-\frac{1}{N}CI+\frac{1}{N}CQ\\
&-\frac{1}{N}IC+\frac{1}{N}\diag(C)-\frac{1}{N}\diag(C)Q\\
&+\frac{1}{N}QC-\frac{1}{N}Q\diag(C)+\diag \big(Q\diag(C)J\big).
\end{split}
\end{align}
We obtain the required equation (\ref{eq:tst}) from (\ref{eq:tst1}) after rearrangement.
The proof is complete.
\end{proof}

We consider some particular cases of (\ref{eq:tst}).

\begin{cor}\label{cor:tst} With respect to the operator $L$ defined by (\ref{eq:T*T}), we have
\begin{align}
L(J)&=J+\frac{2}{N^2}(I-Q),\label{eq:itPi}\\
\begin{split}
L(Q^s)&=\frac{N-2}{N}Q^s+\frac{2}{N}Q^{s+1}-\frac{1}{N}[\diag(Q^s)Q+\frac{1}{N}Q\diag(Q^s)]\\
&\hspace{.5cm}+\frac{1}{N}\diag(Q^s)+\diag\big(Q\diag(Q^s)J\big),\quad \forall\;s\in \Bbb Z_+.\label{eq:itQ}
\end{split}
\end{align}
In particular, if the $s$-step return probabilities $\langle x|Q^s|x\rangle $ do not depend on $x$ for $s\in\Bbb Z_+ $, then (\ref{eq:itQ}) simplifies to
\begin{align}
L(Q^s)=&\frac{N-2}{N}Q^s+\frac{2}{N}Q^{s+1}-\frac{2\tr(Q^s)}{N^2}Q+\frac{2\tr(Q^s)}{N^2}I.\label{eq:itQw}
\end{align}
\end{cor}
\begin{proof}
To see (\ref{eq:itPi}), we apply to (\ref{eq:tst}) the fact that
$QJ=JQ=J$, and obtain the required equation:
\begin{align*}
L(J)&=\frac{N-2}{N}J+\frac{1}{N}2J-\frac{1}{N^2}2Q +\frac{1}{N^2}I+\frac{1}{N^2}I
=J+\frac{2}{N^2}(I-Q).
\end{align*}

Equation (\ref{eq:itQ}) is a straightforward consequence of (\ref{eq:tst}). For the particular case that $\langle x|Q^s|x\rangle $ do not depend on $x$, we note that $\langle x|Q^s|x\rangle$ is equivalent to the arithmetic mean of the diagonal terms of $Q^s$ and so
$\diag(Q^s)=[\tr(Q^s)/N]I$. It follows that (\ref{eq:itQ}) can be written as
\begin{align*}
L(Q^s)
&=\frac{N-2}{N}Q^s+\frac{2}{N}Q^{s+1}-\frac{2\tr(Q^s)}{N^2}Q+\frac{\tr(Q^s)}{N^2}I+\frac{\tr(Q^s)}{N^2}I\\
&=\frac{N-2}{N}Q^s+\frac{2}{N}Q^{s+1}-\frac{2\tr(Q^s)}{N^2}Q+\frac{2\tr(Q^s)}{N^2}I,
\end{align*}
which is (\ref{eq:itQw}).
\end{proof}

Let us turn to some operator norms related to $L$.
We equip $\mathsf M_E$ with the $\ell_1$-norm:
\begin{align}\label{eq:matrix-norm}
\|C\|\triangleq \sum_{x,y\in E}|\langle x|C|y\rangle|
\end{align}
and consider the (induced) operator norms
\[
\vertiii{S}\triangleq \max\{\|S(C)\|;\|C\|=1,C\in \mathsf M_E\}
\]
for linear operators $S:\mathsf M_E\lra \mathsf M_E$. The following proposition will be used in Section~\ref{sec:lap-2} for the proof of Theorem~\ref{thmm:main-1}.

\begin{prop}\label{prop:norm}
We have the following bounds:
\begin{enumerate}
\item [\rm (i)] $\displaystyle \vertiii{L}= 1$,
\item [\rm (ii)] $\displaystyle \vertiii{L-I}\leq \frac{4}{N}$.
\end{enumerate}
\end{prop}
\begin{proof}
We note that by (\ref{eq:tst}),
\begin{align}
L(C)&=\frac{N-1}{N}C+\diag(QCJ)\quad\mbox{if }C=\diag(C),\label{eq:diagL}\\
L(C)&=\frac{N-2}{N}C+\frac{1}{N}(CQ+QC)\quad\mbox{if }\diag(C)=0.\label{eq:offdiagL}
\end{align}
By (\ref{eq:diagL}), we have
\begin{align}
\|L(|x\ra\la x|)\|&=\frac{N-1}{N}+\sum_{y\in E}\sum_{z\in E}\langle y|Q|z\rangle \langle z|x\ra\la x|z\rangle\langle z|J|y\rangle
=1,
\label{eq:diag1}
\end{align}
and by (\ref{eq:offdiagL}), for $x\neq y$, 
\begin{align}\label{Ixy}
\|L(|x\ra \la y|)\|=\frac{N-2}{N}+\frac{1}{N}\sum_{z\in E}\langle y|Q|z\rangle+\frac{1}{N}\sum_{z\in E}\langle z|Q|x\rangle=1.
\end{align}
Since $\big\||x\ra \la y|\big\|= 1$ for all $x,y\in E$, we deduce (i) from (\ref{eq:diag1}) and (\ref{Ixy}).

By (\ref{eq:diagL}) and (\ref{eq:offdiagL}) again, we obtain, respectively,
\begin{align}
\big\|L(|x\ra \la x|)-|x\ra \la x|\big\|=&\sum_{y\in E}\left|\frac{-1}{N}\langle y|x\ra \la x|y\ra+\sum_{z\in E}\la y|Q|z\ra\la z|x\ra \la x|z\ra \la z|J|y\ra \right|\notag\\
=&\sum_{y\in E}\left|\frac{-1}{N}\langle y|x\ra \la x|y\ra+\frac{1}{N}\la x|Q|y\ra\right|\notag\\
=&\frac{2}{N}(1-\langle x|Q|x\rangle)\leq \frac{2}{N}\label{L-IIxx}
\end{align}
and, for $x\neq y$,
\begin{align}
&\big\|L(|x\ra \la y|)-|x\ra \la y|\big\|\notag\\
=&\sum_{z,a\in E}\left|-\frac{2}{N}\la z|x\ra \la y|a\ra+\frac{1}{N}\la z|x\ra\langle y|Q|a\rangle+\frac{1}{N}\la z|Q|x\ra\la y|a\ra\right|\label{L-IIxy-0}\\
=&\frac{1}{N}\big(2-\la y|Q|y\ra-\la x|Q|x\ra\big)+\frac{1}{N}\sum_{a:a\neq y}\langle y|Q|a\ra+\frac{1}{N}\sum_{z:z\neq x}\la z|Q|x\ra\label{L-IIxy-1}\\
= &\frac{2}{N}\big(2-\la y|Q|y\ra-\la x|Q|x\ra\big)\leq \frac{4}{N}, \label{L-IIxy}
\end{align}
where (\ref{L-IIxy-1}) follows by noting that the summands on the right-hand side of (\ref{L-IIxy-0}) are nonzero only if $z=x$ or $a=y$. 
By (\ref{L-IIxx}) and (\ref{L-IIxy}), we deduce (ii).
\end{proof}

\section{Approximating correlations}\label{sec:ptpc} 
The matrices $L^n(J)$ allow for simple expressions in the walk-regular case. This is pointed out in (\ref{eq:Ln-intro}) in Section~\ref{sec:intro}, and we will give the proof later on in Lemma~\ref{lem:alpha}. In general, we can still choose an operator $L_0$ which is similar to $L$ so that $L_0^n(J)$ gives the same expression (\ref{eq:Ln-intro}) (with $L^n(J)$ replaced) even if the assumption of walk-regularity is not in force. One choice is 
to replace $\diag(\,\cdot\,)$ with $[\tr(\,\cdot\,)/N]I$ in the representation (\ref{eq:tst}) of $L$, and consider accordingly
\begin{align}\label{def:L0}
 L_0(C)&\triangleq \frac{N-2}{N}C+\frac{1}{N}(CQ+QC)-\frac{2\tr(C)}{N^2}Q
+\frac{2\tr(C)}{N^2}I
\end{align}
for $C\in \mathsf M_E$. We will work with the matrices $L_0^n(J)$ throughout this section, and it will be helpful to keep in mind that they are meant to ``approximate'' $L^n(J)$, 
in a sense to be made precise in Section~\ref{sec:lp}.

\begin{lem}\label{lem:alpha}
Consider a two-parameter function 
$\alpha:\Bbb Z_+\times \Bbb Z_+\lra \R$
defined recursively as follows. Set $\alpha(0,s)\equiv 0$ for all $s\in \Bbb Z_+$ and, for $n\in \Bbb Z_+$,
\begin{align}
\alpha(n+1,0)&=\frac{2}{N^2}+\alpha(n,0)+\frac{2}{N^2}\sum_{s=1}^\infty \alpha(n,s)\tr(Q^s),\label{rec1}\\
\alpha(n+1,1)&=\frac{-2}{N^2}+\frac{N-2}{N}\alpha(n,1)-\frac{2}{N^2}\sum_{s=1}^\infty \alpha(n,s)\tr(Q^s),\label{rec2}\\
\alpha(n+1,s)&=\frac{N-2}{N}\alpha(n,s)+\frac{2}{N}\alpha(n,s-1),\quad s\geq 2.\label{rec3}
\end{align}
Then the action of the $n$-iteration of the linear map $L_0$ defined by (\ref{def:L0}) on $J$ satisfies
\begin{align}
L_0^n(J)=J+\sum_{s=0}^\infty \alpha(n,s)Q^s,\quad \forall\;n\in \Bbb Z_+.\label{eq:L0n}
\end{align}
In particular, if $Q$ is walk-regular, then
\begin{align}\label{eq:Ln}
L^n(J)=L_0^n(J)=J+\sum_{s=0}^\infty \alpha(n,s)Q^s,\quad \forall\;n\in \Bbb N.
\end{align}
\end{lem}

We note that $\alpha$ has a finite speed of propagation, or more precisely,
$\alpha(n,s)=0$ whenever $s\geq n+1$, and (\ref{rec3}) takes the form of a discrete transport equation.

\begin{proof}[Proof of Lemma~\ref{lem:alpha}]
We prove (\ref{eq:L0n}) by an induction on $n\in\Bbb Z_+$. The case $n=0$ follows plainly by the assumed initial condition of $\alpha$. Suppose that (\ref{eq:L0n}) is true for some $n\in \Bbb Z_+$. Then we have
\begin{align*}
L_0^{n+1}(J)=L_0\big(L^n_0(J)\big)=&L_0\left(J+\sum_{s=0}^\infty \alpha(n,s)Q^s\right)=L_0(J)+\sum_{s=0}^\infty \alpha(n,s)L_0(Q^s),
\end{align*}
where the infinite series are only finite sums since $\alpha$ has a finite speed of propagation. By the definition (\ref{def:L0}) of $L_0$, we can express $L_0(J)$ and $L_0(Q^s)$ on the right-hand side of the foregoing equality by linear combinations of $J$ and $Q^s$ (cf. the proof of Corollary~\ref{cor:tst}), and get
\begin{align*}
L_0^{n+1}(J)
=&\left(J-\frac{2}{N^2}Q+\frac{2}{N^2}I\right)\\
&+\sum_{s=0}^\infty \alpha(n,s)\left(\frac{N-2}{N}Q^s+\frac{2}{N}Q^{s+1}-\frac{2\tr(Q^s)}{N^2}Q+\frac{2\tr(Q^s)}{N^2}I\right)\notag\\
\begin{split}
\hspace{-1cm}=&J+\left(\frac{2}{N^2}+\alpha(n,0)\frac{N-2}{N}+\sum_{s=0}^\infty \alpha(n,s)\frac{2\tr(Q^s)}{N^2}\right)I\\
&+\left(-\frac{2}{N^2}+\alpha(n,1)\frac{N-2}{N}+\alpha(n,0)\frac{2}{N}-\sum_{s=0}^\infty \alpha(n,s)\frac{2\tr(Q^s)}{N^2}\right)Q\\
&+\sum_{s=2}^\infty \left(\alpha(n,s)\frac{N-2}{N}+\alpha(n,s-1)\frac{2}{N}\right)Q^s.
\end{split}
\end{align*}
Since $\tr(I)=N$, (\ref{eq:L0n}) with $n$ replaced by $n+1$ follows from the last equality and the definition of $\alpha(n+1,\,\cdot\,)$ (see (\ref{rec1})--(\ref{rec3})). By mathematical induction (\ref{eq:L0n}) is true for all $n\in \Bbb Z_+$.

The proof of (\ref{eq:Ln}) proceeds with an induction on $n\in \Bbb N$. We need two observations. Notice that when $Q$ is walk-regular,
\begin{align}\label{L=L0}
L(J)=L_0(J)\quad\mbox{ and }\quad L(Q^s)=L_0(Q^s)\quad\forall\; s\in \Bbb Z_+,
\end{align}
by Corollary~\ref{cor:tst} and (\ref{def:L0}).
In addition, it follows from Hoffman's theorem that $J=f(Q)$ for some polynomial $f$ over $\R$ (see \cite[Proposition 3.2]{B:AGT}). We see that
\begin{align}\label{LL0-rest}
\mbox{(\ref{L=L0}) holds}\quad \Longleftrightarrow \quad L=L_0\;\mbox{on}\;\mathcal A\triangleq \{g(Q)\}
\end{align}
where $g$ ranges over all polynomials over $\Bbb C$, and $\mathcal A$ defines an algebra. In addition, it is plain from the definition (\ref{def:L0}) of $L_0$ that
\begin{align}\label{L0-A}
L_0(\mathcal A)\subseteq \mathcal A.
\end{align}

We are ready to give the proof of (\ref{eq:Ln}). It holds for $n=1$ by (\ref{L=L0}). If (\ref{eq:Ln}) holds for some $n\in \Bbb N$, then we get
\[
L^{n+1}(J)=L\big(L_0^n(J)\big)=L_0\big(L_0^n(J)\big)=L_0^{n+1}(J),
\]
where the second equality follows from (\ref{LL0-rest}) and (\ref{L0-A}).
Hence, by mathematical induction, (\ref{eq:Ln}) is true for all $n\in \Bbb N$. The proof is complete.
\end{proof}

Next, we derive the generating function of the above discrete function $\alpha$, using functional calculus for the symmetric matrix $Q$. Recall that we assume the state space of $Q$ has size $N>8$.

\begin{thmm}\label{thmm:GH}
Let $\alpha$ be the two-parameter function defined in Lemma~\ref{lem:alpha} for $L_0$. Write $\mathscr G\alpha(\zeta,q)$ for the two-parameter generating function of $\alpha$:
\begin{align}
\G\alpha(\zeta,q)\triangleq \sum_{n=0}^\infty \zeta^n\sum_{s=0}^\infty q^s\alpha(n,s).
\label{def:A0}
\end{align}
Then $\G\alpha(\zeta,q)$ converges absolutely for any $\zeta,q\in \Bbb C$ with $|\zeta|<1$ and $|q|\leq  1$. Moreover for such $\zeta,q$, 
\begin{align}
\G\alpha(\zeta,q)&=\frac{2\zeta(1-q)[N-\zeta( N-2+2 q)]^{-1}}{N(1-\zeta)^2\tr\left(\frac{1}{N-\zeta(N-2+2Q)}\right)}.\label{eq:A0}
\end{align}
\end{thmm}

Apply the foregoing theorem to (\ref{eq:L0n}) by the spectral representation of $f(Q)$ for complex functions $f$ which are analytic around $[-1,1]$. We obtain the following corollary immediately.

\begin{cor}\label{cor:L0R0-gen}
The generating function of the square matrices $L_0^n(J)$ with respect to exponent $n\in \Bbb Z_+$ is given by
\begin{align}\label{eq:genL}
\begin{split}
\mathscr G L_0 (J)(\zeta)&\triangleq \sum_{n=0}^\infty \zeta^n L_0^n(J)\\
&=\frac{1}{1-\zeta}J+\frac{2\zeta(I-Q)[N-\zeta( N-2+2 Q)]^{-1}}{N(1-\zeta)^2\tr\left(\frac{1}{N-\zeta(N-2+2Q)}\right)}.
\end{split}
\end{align}
Here, the infinite series of matrices converges absolutely with respect to the $\ell_1$-norm on $\mathsf M_E$ (see (\ref{eq:matrix-norm})) whenever $\zeta\in \Bbb C$ with $|\zeta|<1$.
\end{cor}

The rest of this section is devoted to the proof of Theorem~\ref{thmm:GH}. 
The conclusion will be given at the end of this section.

We begin with an a priori estimate to ensure that the radius of convergence for $\G\alpha$ is not degenerate. See also Lemma~\ref{lem:alpha-growth} for different estimates.

\begin{lem}\label{lem:bdd-0}
The entries of $\alpha$ satisfy
\begin{align}\label{ineq:radii-a}
\max\left\{|\alpha(n,0)|,|\alpha(n,1)|,\sum_{s=2}^\infty|\alpha(n,s)|\right\}\leq \frac{1}{N}\left(1+\frac{6}{N}\right)^{n-1},\quad n\in \Bbb N.
\end{align}
\end{lem}
\begin{proof}
Note that $|\tr(Q^s)/N|\leq 1$ for all $s\in \Bbb Z_+$. Hence, (\ref{rec1})--(\ref{rec3}) imply 
\begin{align*}
&|\alpha(n+1,0)|\leq \frac{2}{N^2}+|\alpha(n,0)|+\frac{2}{N}|\alpha(n,1)|+\frac{2}{N}\sum_{s=2}^\infty |\alpha(n,s)|,\\
&|\alpha(n+1,1)|\leq \frac{2}{N^2}+|\alpha(n,1)|+\frac{2}{N}\sum_{s=2}^\infty |\alpha(n,s)|,\\
&\sum_{s=2}^\infty |\alpha(n+1,s)|\leq \sum_{s=2}^\infty |\alpha(n,s)|+\frac{2}{N}|\alpha(n,1)|.
\end{align*}
Then (\ref{ineq:radii-a}) can be checked by induction on $n$ (recall the initial condition of $\alpha$) and the foregoing inequalities.
The proof is complete.
\end{proof}

\begin{lem}\label{lem:solve}
Let $\zeta,q\in \Bbb C$ be such that  
\begin{align}
|\zeta|<\frac{N}{N+6}\quad\mbox{ and }\quad |q|\leq 1.\label{zetaq}
\end{align}
Then the infinite series (\ref{def:A0}) defining $\mathscr G\alpha(\zeta,q)$ converges absolutely, and 
we have
\begin{align}
&\sum_{n=0}^\infty \zeta^n\alpha(n,0)=\frac{-[N-\zeta(N-2)]}{N(1-\zeta)}\times\sum_{n=0}^\infty \zeta^n\alpha(n,1),\label{alpha0-gen}\\
&\sum_{n=0}^\infty \zeta^n\alpha(n,1)=\frac{-2\zeta}{(1-\zeta)[N-\zeta(N-2)]^2\tr\left(\frac{1}{N-\zeta(N-2+2Q)}\right)},\label{alpha1-gen}\\
&\sum_{n=0}^\infty \zeta^n\sum_{s=2}^\infty q^s\alpha(n,s)=\frac{2\zeta q^2}{N-\zeta(N-2+2q)}\times
\sum_{n=0}^\infty \zeta^n\alpha(n,1).\label{alpha-gen}
\end{align}
Hence, (\ref{eq:A0}) holds for $\zeta,q\in \Bbb C$ satisfying (\ref{zetaq}).
\end{lem}
\begin{proof}
It follows immediately from Lemma~\ref{lem:bdd-0} that $\G \alpha(\zeta,q)$ converges absolutely for $\zeta,q\in \Bbb C$ satisfying (\ref{zetaq}). We may assume throughout the proof that $\zeta\neq 0$.

First, we derive (\ref{alpha0-gen}) and (\ref{alpha-gen}) in order. We add up both sides of (\ref{rec1}) and (\ref{rec2})
and get
\[
\alpha(n+1,0)+\alpha(n+1,1)=\alpha(n,0)+\frac{N-2}{N}\alpha(n,1),\quad n\in \Bbb Z_+.
\]
Using the foregoing equality and the fact that $\alpha(0,\cdot)\equiv 0$, we deduce that
\begin{align*}
\frac{1}{\zeta}\sum_{n=0}^\infty \zeta^n\alpha(n,0)+\frac{1}{\zeta}\sum_{n=0}^\infty \zeta^n\alpha(n,1)&=\sum_{n=0}^\infty \zeta^{n}\alpha(n+1,0)+\sum_{n=0}^\infty \zeta^{n}\alpha(n+1,1)
\\
&=\sum_{n=0}^\infty \zeta^n\alpha(n,0)+\frac{N-2}{N}\sum_{n=0}^\infty \zeta^n\alpha(n,1),
\end{align*}
which implies (\ref{alpha0-gen}).
Next, by (\ref{rec3}) and the fact that $\alpha(0,\,\cdot\,)\equiv 0$, we get
\begin{align*}
&\frac{1}{\zeta}\sum_{n=0}^\infty \zeta^n\sum_{s=2}^\infty q^s \alpha(n,s)=
\sum_{n=0}^\infty \zeta^n\sum_{s=2}^\infty q^s\alpha(n+1,s)\\
&\hspace{.5cm}=\left(\frac{N-2}{N}+\frac{2q}{N}\right)\sum_{n=0}^\infty\zeta^n \sum_{s=2}^\infty q^s\alpha(n,s)+\frac{2q^2}{N}\sum_{n=0}^\infty \zeta^n\alpha(n,1),
\end{align*}
and the equality (\ref{alpha-gen}) follows.

Let us solve for the series
$\sum_{n=0}^\infty \zeta^n\alpha(n,1)$ which appears on both of the right-hand sides of (\ref{alpha0-gen}) and (\ref{alpha-gen}). By the initial condition $\alpha(0,\,\cdot\,)\equiv 0$ and (\ref{rec2}),
\begin{align*}
\begin{split}
&\frac{1}{\zeta}\sum_{n=0}^\infty \zeta^n\alpha(n,1)
=\sum_{n=0}^\infty \zeta^n\alpha(n+1,1)\\
&\hspace{1cm}=\frac{-2}{N^2(1-\zeta)}
+\frac{N-2}{N}\sum_{n=0}^\infty \zeta^n\alpha(n,1)
-\frac{2\tr(Q )}{N^2}\sum_{n=0}^\infty \zeta^n\alpha(n,1)\\
&\hspace{1.5cm}-\frac{2}{N^2}\tr\left(\sum_{n=0}^\infty \zeta^n\sum_{s=2}^\infty Q^s\alpha(n,s)\right),
\end{split}
\end{align*}
where the absolute convergence of the infinite series is justified by Lemma~\ref{lem:bdd-0}.
By the foregoing equality, we get
\begin{align*}
&\sum_{n=0}^\infty \zeta^n\alpha(n,1)\\
=&\frac{-2\zeta}{N^2(1-\zeta)}+\frac{\zeta (N-2)}{N}\sum_{n=0}^\infty \zeta^n\alpha(n,1)-\frac{2\zeta\tr(Q )}{N^2}\sum_{n=0}^\infty \zeta^n\alpha(n,1)\\
&-\frac{2\zeta}{N^2}\tr\left(\sum_{n=0}^\infty \zeta^n\sum_{s=2}^\infty Q^s\alpha(n,s)\right)\\
=&\frac{-2\zeta}{N^2(1-\zeta)}+\left[\frac{\zeta (N-2)}{N}- \frac{2\zeta\tr(Q )}{N^2}- \frac{4\zeta^2}{N^2}\tr\left( \frac{Q^2}{N-\zeta(N-2+2Q)}
\right)\right]\\
&\times \sum_{n=0}^\infty \zeta^n\alpha(n,1),
\end{align*}
where the second equality follows from (\ref{alpha-gen}), and hence,
\begin{align}\label{alphan1}
\begin{split}
&\sum_{n=0}^\infty \zeta^n\alpha(n,1)\\
&\hspace{.5cm}=\frac{-2\zeta}{N(1-\zeta)\left[N-N\zeta+2\zeta+2\zeta\frac{\tr(Q)}{N}+\frac{4\zeta^2}{N}  \tr\left( \frac{Q^2}{N-\zeta(N-2+2Q)}
\right)\right]}.
\end{split}
\end{align}
By linearity of trace, we obtain 
\begin{align*}
&N-N\zeta+2\zeta+2\zeta\frac{\tr(Q)}{N}+\frac{4\zeta^2}{N}  \tr\left( \frac{Q^2}{N-\zeta(N-2+2Q)}
\right)\\
=&\frac{[N(1-\zeta)+2\zeta]^2}{N}\tr\left(\frac{1}{N-\zeta(N-2+2Q)}\right),
\end{align*}
and hence,
(\ref{alpha1-gen}) by the foregoing equality and (\ref{alphan1}).

Finally, we recall (\ref{def:A0}) and use (\ref{alpha0-gen})--(\ref{alpha-gen}) to obtain
\begin{align*}
\G\alpha(\zeta,q)&=\left\{\frac{-[N-\zeta(N-2)]}{N(1-\zeta)}+q+\frac{2\zeta q^2}{N-\zeta(N-2+2q)}\right\}
\\
&\hspace{1cm}\times\frac{-2\zeta}{(1-\zeta)[N-\zeta(N-2)]^2\tr\left(\frac{1}{N-\zeta(N-2+2Q)}\right)},
\end{align*}
from which we deduce the equality (\ref{eq:A0}) for $ \zeta,q\in \Bbb C$ satisfying (\ref{zetaq}). The proof is complete.
\end{proof}

From now on, we write $C(0,R)$, $D(0,R)$ and $\overline{D}(0,R)$ for the circle, the open disc and the close disc, respectively, centered at $0$ with radius $R\in (0,\infty)$ in the complex plane.
We study the denominator of the function  on the right-hand side of (\ref{eq:A0}).

\begin{lem}\label{lem:bdd-1}
For any $\vep\in [0,N)$, 
\begin{align}
\begin{split}\label{ineq:Mobm}
&\min_{\zeta\in C(0,1-\frac{\vep}{N})}
\left|\frac{1}{N}\tr\left(\frac{1-\zeta}{N-\zeta(N-2+2Q)}\right)\right|\\
&\hspace{4cm}=\frac{1}{N^2}\tr\Bigg(\frac{\vep}{2-2Q+\vep-\frac{\vep}{N}(2-2Q)}\Bigg),\\
\end{split}\\
\begin{split}
&\max_{\zeta\in C(0,1-\frac{\vep}{N})}
\left|\frac{1}{N}\tr\left(\frac{1-\zeta}{N-\zeta(N-2+2Q)}\right)\right|\\
&\hspace{4cm}= \frac{1}{N^2}\tr\Bigg(\frac{2-\frac{\vep}{N}}{
1+\big(1-\frac{\vep}{N}\big)\big(1-\frac{2}{N}+\frac{2Q}{N}\big)}\Bigg).\label{ineq:MobM}
\end{split}
\end{align}
Here, meromorphic functions are defined at their removable singularities in the natural way,  and the right-hand side of the equality in (\ref{ineq:Mobm}) is read as $1/N^2$ if $\vep=0$ (since $1$ is an eigenvalue of $Q$).
Moreover, in the above display, the minimum and the maximum are attained at $1-\frac{\vep}{N}$ and $-(1-\frac{\vep}{N})$, respectively. 
\end{lem}

\begin{proof}
For every $q\in [-1,1]$, consider the M\"obius transformation 
\[
M_q(\zeta)\triangleq  \frac{1-\zeta}{N-\zeta(N-2+2q)}.
\]
Note that $M_1$ is just the constant map $1/N$.

Let us make some observations for $M_q$, when $q\in [-1,1)$.
By a standard result of M\"obius transformations, $M_q$ maps $C(0,1-\frac{\vep}{N})$ to a nondegenerate circle, say $C_q$, in $\Bbb C$ because it is nonconstant and analytic in an open set containing $\overline{D}(0,1)$.
The circle $C_q$ is symmetric about the real line because $M_q$ is defined by real coefficients, and plainly intersects the real line at 
\begin{align}\label{+1}
\displaystyle  M_q\left(1-\frac{\vep}{N}\right)=\frac{\vep/N}{2-2q+\vep-\frac{\vep}{N}(2-2q)}
\end{align}
and
\begin{align}\label{-1}
M_q\left(-1+\frac{\vep}{N}\right)=\frac{2-\frac{\vep}{N}}{
N\big[1+\big(1-\frac{\vep}{N}\big)\big(1-\frac{2}{N}+\frac{2q}{N}\big)\big]},
\end{align}
which are distinct strictly positive  real numbers. 
This means that $C_q$ is contained in the half plane $\{\zeta\in \Bbb C;\Re(\zeta)>0\}$. In addition, note that for each $q\in [-1,1)$, the value in (\ref{+1}) is strictly less than the value in (\ref{-1}), and $M_q(\sqrt{-1})$ has strictly negative imaginary part. We deduce that 
for $q\in [-1,1)$ and $\zeta=(1-\frac{\vep}{N})e^{\sqrt{-1}\theta}$ with $\theta\in [0,2\pi]$,
\begin{align}\label{Mq:rotat}
M_q(\zeta)=a_q-b_qe^{\sqrt{-1}\theta}
\end{align}
for some $a_q,b_q\in (0,\infty)$ independent of $\theta$. The foregoing equality holds trivially for $q=1$ if we set $a_1=1/N$ and $b_1=0$.

By the special form (\ref{Mq:rotat}) of $M_q$ for all $q\in [-1,1]$, the optimization problems in (\ref{ineq:Mobm}) and (\ref{ineq:MobM}) take the forms of minimizing and maximizing $|A-Be^{\sqrt{-1}\theta}|$ subject to $\theta\in [0,2\pi]$, for fixed $A,B\in (0,\infty)$. For the latter two, solutions are given by $\theta=0$ and $\theta=\pi$, respectively. 
Hence, the equalities in (\ref{ineq:Mobm}) and (\ref{ineq:MobM}) follow upon using (\ref{+1}) and (\ref{-1}). 
\end{proof}

\begin{rmk}
Note that \cite[Theorem 1]{P:ZPS} studies locations of zeros for complex functions taking the form $\zeta\lmt P(\zeta)+\zeta^k\int_{[0,1]}(1-\zeta q)^{-1}\mu(dq)$, which arises from the investigation of Riesz summability. Here, $P$ are polynomials of degree $k-1$ for $k\in \Bbb Z_+$ ($P(\zeta)\equiv 0$ if $k=0$) and $\mu$ is a finite measure on $[0,1]$.
The context in \cite[Theorem 1]{P:ZPS} overlaps in part the context of Lemma~\ref{lem:bdd-1}. \qed 
\end{rmk}

\paragraph{\bf Conclusion for the proof of Theorem~\ref{thmm:GH}.}
We have seen that (\ref{eq:A0}) holds when $\zeta,q\in \Bbb C$ satisfying (\ref{zetaq}), by Lemma~\ref{lem:solve}. By a standard result of several complex variables (see, e.g., H\"ormander~\cite[Theorem 2.2.1]{H:SCV}) and
 Lemma~\ref{lem:bdd-1}, the function in two complex variables on the right-hand side of (\ref{eq:A0}) is analytic in the open polydisc $D=\{(\zeta,q)\in \Bbb C\times \Bbb C;|\zeta|<1,|q|<1\}$. Hence, by Cauchy's inequalities for analytic functions in several complex variables \cite[Theorem~2.2.7]{H:SCV}, $\G\alpha(\zeta,q)$ converges absolutely on the polydisc $D$. Next, if we fix $\zeta\in \Bbb C$ such that $|\zeta|<1$ and repeat the above argument with respect to the single complex variable $q$, then the assertion of Theorem~\ref{thmm:GH} can be extended up to the boundary case $|q|=1$. The proof is complete. \qed

\section{Laplace transforms of first-order meeting times}\label{sec:lp}

\subsection{Meeting time distributions, voter correlations and approximating correlations}\label{sec:lap-1}
In this section, 
we derive an infinite series expression for the Laplace transform of $M_{U,V}$ in terms of the true correlations
\[
\frac{1}{N}\E[\langle \beta_u|L^n(J)|\beta_u\rangle]=\E_{\beta_u}[\xi_n(U)\xi_n(V_\infty)]
\]
(recall (\ref{eq:coupling})),
and compute the analogous series in terms of the approximating correlations $\frac{1}{N}\E[\langle \beta_u|L^n_0(J)|\beta_u\rangle]$.
The results of the present section will be applied in Section~\ref{sec:lap-2} for the proof of Theorem~\ref{thmm:main-1-0}.

From now on, we write $S_\xi(C)\equiv \la \xi|C|\xi\rangle$ for all deterministic configurations $\xi$, and in addition,
\begin{align}\label{def:Sbetau}
S_{\beta_u}(C)\equiv \E[\langle \beta_u|C|\beta_u\rangle],\quad C\in\mathsf M_E,
\end{align}
for the Bernoulli configurations $\beta_u=\sum_x \beta_u(x)|x\ra$ (recall that $\beta_u(x)$ are i.i.d. Bernoulli with mean $u$).

\begin{lem}\label{lem:lap-ac-1}
For every $\lambda\in (0,\infty)$ and $\xi\in \{1,0\}^E$,
\begin{align}\label{BuLn-2}
\begin{split}
&\frac{\lambda}{N+\lambda}\sum_{n=0}^\infty\left(\frac{N}{N+\lambda}\right)^n \frac{1}{N}S_\xi L_0^n(J)
=p_1(\xi)^2\\
&\hspace{4cm}+\frac{1}{N}S_\xi\left(\frac{2(I-Q)(\lambda+2-2Q)^{-1}}{ \lambda\tr\left(\frac{1}{\lambda+2-2Q}\right)}\right),
\end{split}
\end{align}
where the series on the left-hand side converges absolutely. 
\end{lem}
\begin{proof}
We apply $\frac{\lambda}{N+\lambda}\times \frac{1}{N}S_\xi$ to both sides of (\ref{eq:genL}) with $\zeta$ set to be $N/(N+\lambda)$. Then notice that
\begin{align*}
\frac{1}{N}S_\xi(J)=\frac{\langle \xi|\1\rangle \langle \1|\xi\rangle}{N^2}=p_1(\xi)^2
\end{align*}
and the second term on the right-hand side of (\ref{eq:genL}) with $\zeta$ set to be $N/(N+\lambda)$ becomes
\begin{align*}
\frac{2(N+\lambda) (I-Q)(\lambda+2-2Q)^{-1}}{\lambda^2\tr\left(\frac{1}{\lambda+2-2Q}\right)} .
\end{align*}
The equality (\ref{BuLn-2}) now follows plainly.
\end{proof}

Next, we relate the Laplace transform of $M_{U,V}$ to the generating function of $\frac{1}{N}S_{\beta_u}L_0^n(J)$ (in $n$) in the following lemma. We define the mean local density $p_{10}(\xi)$ of configuration $\xi$ by 
\begin{align}
p_{10}(\xi)\triangleq \sum_{x,y\in E}\frac{1}{N}\langle x|Q|y\rangle \xi(x)\xi(y)= \frac{1}{N}\langle \xi|Q|\xi\rangle.\label{def:p10}
\end{align}

\begin{lem}\label{lem:MUV-exp}
\begin{enumerate}
\item [(\rm i)] For every $\xi\in \{1,0\}^E$ and $n\in \Bbb Z_+$,
\[
\frac{1}{N}S_\xi L^{n+1}(J)-\frac{1}{N}S_\xi L^n(J)=\frac{2}{N^2}\E_\xi[p_{10}(\xi_n)].
\]

\item [\rm (ii)]For every $\lambda\in (0,\infty)$ and $u\in (0,1)$,
\begin{align*}
\begin{split}
&\E\big[e^{-\lambda M_{U,V}}\big]\\
&=1-\frac{\lambda N^2}{2u(1-u)(N+\lambda)}\sum_{n=0}^\infty \left(\frac{N}{N+\lambda}\right)^n\left(\frac{1}{N}S_{\beta_u}L^{n+1}(J)-\frac{1}{N}S_{\beta_u}L^n(J)  \right),
\end{split}
\end{align*}
where the series on the right-hand side converges absolutely.

\item [\rm (iii)] For every $\lambda\in (0,\infty)$ and $u\in (0,1)$,
\begin{align*}
\begin{split}
&\left.\E\left[\int_0^\infty e^{-\lambda t/2}\1_{\{V\}}(X^U_t)dt\right]\right/\Bbb E\left[\int_0^\infty e^{-\lambda t/2}\1_{\{V\}}(X^V_t)dt\right]\\
&=1-\frac{\lambda N^2}{2u(1-u)(N+\lambda)}\sum_{n=0}^\infty \left(\frac{N}{N+\lambda}\right)^n\left(\frac{1}{N}S_{\beta_u}L_0^{n+1}(J)-\frac{1}{N}S_{\beta_u}L_0^n(J)\right),
\end{split}
\end{align*}
where the series on the right-hand side converges absolutely.
\end{enumerate}
\end{lem}
\begin{proof}
The proof of (i) follows immediately from the fact: for every $\xi\in \{1,0\}^E$ and $n\in \Bbb Z_+$,
\begin{align}\label{p1:timeevolution}
\frac{1}{N}S_{\xi}L^n(J)=\frac{1}{N}S_\xi(J)+\frac{2}{N^2}\sum_{j=0}^{n-1}\E_\xi[p_{10}(\xi_j)],
\end{align}
where $p_{10}(\xi)$ is defined by (\ref{def:p10}).
The above equation is implicit in the proof of \cite[Proposition~3.1 (iii)]{CCC:WF} since $\frac{1}{N}S_\xi L^n(J)\equiv \E_\xi[\xi_n(U)\xi_n(V_\infty)]$.

To prove (ii), we resort to the duality equation (\ref{eq:MTVM})
and use the fact that $(\xi_t)$ is equal to $(\xi_n)$ time-changed by an independent Poisson process with rate $N$. We get
\begin{align*}
1-\E[e^{-\lambda M_{U,V}}]&=\int_0^\infty \lambda e^{-\lambda s}\frac{\E_{\beta_u}[p_{10}(\xi_s)]}{u(1-u)}ds\\
&\hspace{-2cm}=\frac{1}{u(1-u)}\frac{\lambda}{N+\lambda}\sum_{n=0}^\infty \left(\frac{N}{N+\lambda}\right)^n\E_{\beta_u}[p_{10}(\xi_n)]\\
&\hspace{-2cm}=\frac{\lambda N^2}{2u(1-u)(N+\lambda)}\sum_{n=0}^\infty \left(\frac{N}{N+\lambda}\right)^n \left(\frac{1}{N}S_{\beta_u}L^{n+1}(J)-\frac{1}{N}S_{\beta_u}L^n(J)\right),
\end{align*}
where the last equality follows from (i).

It remains to prove (iii). On one hand, we note that
\begin{align*}
&\left.\E\left[\int_0^\infty e^{-t\lambda/2}\1_{\{V\}}(X^{U}_t)dt\right]\right/ \E\left[\int_0^\infty e^{-\lambda t/2}\1_{\{V\}}(X^{V}_t)dt\right]\\
&= \left.\frac{1}{N}\sum_{x,y\in E}\langle x|Q|y\rangle \left\langle x\left|\int_0^\infty e^{-t\lambda/2+t(Q-I)}dt\right|y\right\rangle\right/\frac{1}{N}\sum_{y\in E}\left\langle y\left|
\int_0^\infty e^{-t\lambda/2+t(I-Q)}dt
\right|y\right\rangle\\
&=\left.\frac{1}{N}\tr\Bigg(\frac{Q}{\frac{\lambda}{2}+(I-Q)}\Bigg)\right/\frac{1}{N}\tr\Bigg(\frac{1}{\frac{\lambda}{2}+(I-Q)}\Bigg)\\
&=\left.\tr\Bigg(\frac{Q}{\lambda+2(I-Q)}\Bigg)\right/\tr\Bigg(\frac{1}{\lambda+2(I-Q)}\Bigg),
\end{align*}
and so (iii) follows if we can show that
\begin{align*}
&\left.\tr\Bigg(\frac{Q}{\lambda+2(I-Q)}\Bigg)\right/\tr\Bigg(\frac{1}{\lambda+2(I-Q)}\Bigg)\\
=&1-\frac{\lambda N^2}{2u(1-u)(N+\lambda)}\sum_{n=0}^\infty \left(\frac{N}{N+\lambda}\right)^n\left(\frac{1}{N}S_{\beta_u}L_0^{n+1}(J)-\frac{1}{N}S_{\beta_u}L_0^n(J)\right).
\end{align*}
The foregoing equality, however, is a consequence of (\ref{BuLn-2}) if we apply a randomization by $\beta_u$ to both sides of (\ref{BuLn-2}) and notice \begin{align}
\begin{split}\label{Sbetau-0}
&\frac{1}{N}S_{\beta_u}\Bigg(\frac{2(I-Q)(\lambda+2-2Q)^{-1}}{\lambda \tr\left(\frac{1}{\lambda+2-2Q}\right)}\Bigg)\\
&\hspace{1.8cm}=\left.\left[2(u-u^2)\tr\left(\frac{1-Q}{\lambda+2-2Q}\right)\right]\right/\left[\lambda N \tr\left(\frac{1}{\lambda+2-2Q}\right)\right].
\end{split}
\end{align}

To see (\ref{Sbetau-0}), we first compute
\begin{align}
S_{\beta_u}(Q^s)&=\E\left[\sum_{x,y\in E}\beta(x)\langle x|Q^s|y\rangle\beta(y)\right]\notag\\
&=\sum_{\stackrel{\scriptstyle  x,y\in E}{x\neq y}}\langle x|Q^s|y\rangle u^2+\sum_{x\in E}\langle x|Q^s|x\rangle u\notag\\
&=u^2N+(u-u^2)\tr(Q^s),\quad \forall\;s\in \Bbb Z_+,\label{Sbetau}
\end{align}
and similarly, $S_{\beta_u}(J)=(N-1)u^2+u$. Hence, we get
\begin{align}\label{SSSSS}
S_{\beta_u}\big(Q^s\1_{(-\infty,1)(Q)}\big)
&=S_{\beta_u}\left(Q^s-J\right)
=(u-u^2)\tr\big(Q^s;Q<1\big).
\end{align}
Here, we write 
\begin{align}\label{eq:trace-set}
\tr\big(f(Q);Q\in A\big)\triangleq \tr\big(f(Q)\1_{A}(Q)\big),\quad A\subseteq \R.
\end{align}
Then by polynomial approximation, (\ref{SSSSS}) implies
\begin{align}
S_{\beta_u}\big((I-Q)(\lambda+2-2Q)^{-1}\big)&=S_{\beta_u}\big((I-Q)(\lambda+2-2Q)^{-1}\1_{(-\infty,1)}(Q)\big)\notag\\
&=(u-u^2)
\tr\left(\frac{1-Q}{\lambda+2-2Q};Q<1\right)\notag\\
&=(u-u^2)
\tr\left(\frac{1-Q}{\lambda+2-2Q}\right).\label{Subeta-2}
\end{align}
The equality (\ref{Subeta-2}) is enough to get (\ref{Sbetau-0}), and the proof is complete.
\end{proof}

\begin{cor}\label{cor:MUV-wr}
Suppose that $Q$ is walk-regular. Then
\begin{align}
\begin{split}
&\E[e^{-\lambda M_{U,V}}]\\
&\hspace{.5cm}=\left.\E\left[\int_0^\infty e^{-\lambda t/2}\1_{\{V\}}(X^U_t)dt\right]\right/
\E\left[\int_0^\infty e^{-\lambda t/2}\1_{\{V\}}(X^V_t)dt\right].\label{MUV-1}
\end{split}
\end{align}
\end{cor}
\begin{proof}
Equation (\ref{MUV-1}) follows at once if we compare (ii) and (iii) of Lemma~\ref{lem:MUV-exp}, since
\[
\frac{1}{N}S_{\beta_u}L^n(J)=\frac{1}{N}S_{\beta_u}L_0^n(J),\quad \forall\;n\in \Bbb Z_+,
\]
by (\ref{eq:Ln}) in Lemma~\ref{lem:alpha}.
The proof is complete.
\end{proof}

\subsection{Approximation}\label{sec:lap-2} 
This section is devoted to the proof of Theorem~\ref{thmm:main-1-0}. 
Recall the notation $S_\xi$ and $S_{\beta_u}$ defined at the beginning of Section~\ref{sec:lap-1}. To prove Theorem~\ref{thmm:main-1}, it is enough to bound
\begin{align}
\begin{split}\label{T}
&\frac{\lambda N^2}{2u(1-u)(N+\lambda)}
\sum_{n=0}^\infty \left(\frac{N}{N+\lambda}\right)^n\left(\frac{1}{N}S_{\beta_u}L^{n+1}(J)-\frac{1}{N}S_{\beta_u}L^n(J)\right)\\
&-\frac{\lambda N^2}{2u(1-u)(N+\lambda)}\sum_{n=0}^\infty \left(\frac{N}{N+\lambda}\right)^n\left(\frac{1}{N}S_{\beta_u}L_0^{n+1}(J)-\frac{1}{N}S_{\beta_u}L_0^n(J)\right)
\end{split}
\end{align}
by Lemma~\ref{lem:MUV-exp} (ii) and (iii).
For fixed $m\in \Bbb N$,
we will handle separately and in order the objects: (1) the partial sum of the first term in (\ref{T}) for $n$ ranging from $mN+1$ to $\infty$,
(2) the partial sum of the second term in (\ref{T}) for $n$ ranging from $mN+1$ to $\infty$, and (3) the difference in (\ref{T}) with the upper limits $\infty$ of the two series replaced with $mN$.

The first object described above is easy to deal with.

\begin{lem}\label{lem:T1}
For every $\lambda\in(0,\infty)$, $u\in (0,1)$ and $m\in \Bbb N$,
\begin{align*}
&\left|
\frac{\lambda N^2}{2u(1-u)(N+\lambda)}
\sum_{n=mN+1}^{\infty}\left(\frac{N}{N+\lambda}\right)^n\left(\frac{1}{N}S_{\beta_u}L^{n+1}(J)-\frac{1}{N}S_{\beta_u}L^n(J)\right)\right|\\
&\hspace{8cm}\leq \frac{1}{u(1-u)}\left(\frac{N}{N+\lambda}\right)^{mN}.
\end{align*}
\end{lem}
\begin{proof}
By Lemma~\ref{lem:MUV-exp} (i), 
\[
\frac{N^2}{2}\left(\frac{1}{N}S_{\beta_u}L^{n+1}(J)-\frac{1}{N}S_{\beta_u}L^n(J)\right)=\E_{\beta_u}[p_{10}(\xi_n)]\in [0,1].
\]
Hence,
\begin{align*}
&\left|\frac{\lambda N^2}{2u(1-u)(N+\lambda)}\sum_{n=mN}^\infty \left(\frac{N}{N+\lambda}\right)^n\left(\frac{1}{N}S_{\beta_u}L^{n+1}(J)-\frac{1}{N}S_{\beta_u}L^n(J)\right)\right|\\
&\hspace{.5cm}\leq \frac{\lambda}{u(1-u)(N+\lambda)}\sum_{n=mN}^\infty 
\left(\frac{N}{N+\lambda}\right)^n=\frac{1}{u(1-u)}\left(\frac{N}{N+\lambda}\right)^{mN},
\end{align*}
as required.
\end{proof}

To bound the second object described below (\ref{T}), we turn to the asymptotics of the discrete function $\alpha$ defined in Lemma~\ref{lem:alpha}.

\begin{lem}\label{lem:alpha-growth}
For every $\vep\in (0,1]$ and $n\in \Bbb N$, we have
\begin{align}
&|\alpha(n,0)|\leq \frac{(\vep^{-1}+4\sqrt{2})}{\pi}\times \frac{4+\vep}{\vep} \times\frac{1}{N}\left(\frac{N}{N-\vep}\right)^{\max\{n-1,0\}}, \label{alpha0-exp-bdd}\\
&|\alpha(n,1)|\leq \frac{9}{\pi}\times \frac{4+\vep}{\vep} \times\frac{1}{N}\left(\frac{N}{N-\vep}\right)^{\max\{n-1,0\}},\label{alpha1-exp-bdd}\\
\begin{split}
&\sup_{q\in [-1,1]}\left|\sum_{s=2}^\infty q^s\alpha(n,s)\right|\leq \frac{1}{\pi}\left(\vep^{-1}+\frac{16\sqrt{2}}{(1-\cos 1)^{1/2}}\right)\times \frac{4+\vep}{\vep}\\
&\hspace{6cm}\times\frac{1}{N}\left(\frac{N}{N-\vep}\right)^{\max\{n-2,0\}}.\label{alpha-exp-bdd}
\end{split}
\end{align}
\end{lem}
\begin{proof}
Throughout this proof, we fix $\vep\in (0,1]$. We start with the proof of (\ref{alpha0-exp-bdd}), and fix $n\geq 1$.
For $\zeta=(1-\frac{\vep}{N})z$, it follows from (\ref{alpha0-gen}), (\ref{alpha1-gen}) and Cauchy's integral formula that
\begin{align}
\begin{split}\label{Gamma-goal}
&\alpha(n,0)=\frac{1}{2\pi \sqrt{-1}}\int_{C(0,1)}\frac{1}{z^{n+1}}\\
&\hspace{2cm}\times\frac{2z}{[N-N(1-\frac{\vep}{N})z][N-(1-\frac{\vep}{N})z(N-2)]}\\
&\hspace{2.5cm}\times \frac{1}{
\frac{1}{N}\tr\left(\frac{1-(1-\frac{\vep}{N})z}{N-(1-\frac{\vep}{N})z(N-2+2Q)}\right)}dz\times \frac{1}{N}\left(1-\frac{\vep}{N}\right)^{-n+1}.
\end{split}
\end{align}
In the following, we bound
\begin{align}\label{Gamma}
\begin{split}
&\Bigg|\frac{1}{2\pi \sqrt{-1}}\int_{\Gamma_j}\frac{1}{z^{n+1}}\times\frac{2z}{[N-N(1-\frac{\vep}{N})z][N-(1-\frac{\vep}{N})z(N-2)]}\\
&\hspace{5.5cm}\times \frac{1}{\frac{1}{N}
\tr\left(\frac{1-(1-\frac{\vep}{N})z}{N-(1-\frac{\vep}{N})z(N-2+2Q)}\right)}dz\Bigg|,
\end{split}
\end{align}
for $j=1,2$,
where the arcs $\Gamma_j$ are defined by
\begin{align}
\Gamma_1&\triangleq \{z\in C(0,1)\setminus \{-1\};|\arg(z)|\leq 1/N\},\label{Gamma1}\\
\Gamma_2&\triangleq \{z\in C(0,1)\setminus \{-1\};|\arg(z)|\in (1/N,\pi)\}\cup \{-1\}\label{Gamma2}
\end{align}
and their disjoint union is equal to $C(0,1)$.
Here for any $z\in \Bbb C\setminus (-\infty,0]$, we write $\arg(z)=\theta$ for $z=|z|e^{i\theta}$ for $\theta\in (-\pi,\pi)$.

We handle (\ref{Gamma}) for $j=1$ first. Using (\ref{ineq:Mobm}), we see that
\begin{align}\label{Mobm-bdd}
\min_{\zeta\in C(0,1-\frac{\vep}{N})}
\left|\frac{1}{N}\tr\left(\frac{1-\zeta}{N-\zeta(N-2+2Q)}\right)\right|\geq \left(\frac{\vep}{4+\vep}\right)\frac{1}{N}.
\end{align}
By (\ref{Mobm-bdd}) and the fact that
$\min_{z\in C(0,1)}\big|1-az\big|=1-a$ for any $a\in (0,1)$,  we have
\begin{align}
&\Bigg|\frac{1}{2\pi \sqrt{-1}}\int_{\Gamma_1}\frac{1}{z^{n+1}}\times\frac{2z}{[N-N(1-\frac{\vep}{N})z][N-(1-\frac{\vep}{N})z(N-2)]}\notag\\
&\hspace{5.5cm}\times \frac{1}{
\frac{1}{N}\tr\left(\frac{1-(1-\frac{\vep}{N})z}{N-(1-\frac{\vep}{N})z(N-2+2Q)}\right)}dz\Bigg|\notag\\
&\hspace{.5cm}\leq \frac{1}{2\pi} \int_{-1/N}^{1/N}\frac{2}{\vep\left(2+\vep-\frac{2\vep}{N}\right)\times \vep/(4+\vep)\times (1/N)} d\theta\leq \frac{1}{\vep \pi }\times \frac{4+\vep}{\vep}.\label{Gamma1-1}
\end{align}

Next, we bound (\ref{Gamma}) for $j=2$.
Note that
\begin{align*}
\forall\; \delta\in (0,1)\;\mbox{and}\; \theta\in [0,2\pi],\quad |1-(1-\delta)e^{\sqrt{-1}\theta}|
\geq 
&\sqrt{2(1-\delta)}\sqrt{1-\cos \theta}.
\end{align*}
Applying the foregoing inequality to
\begin{align*}
&\frac{1}{N}\left|N-N\left(1-\frac{\vep}{N}\right)z\right|=\left|1-\left(1-\frac{\vep}{N}\right)\right|
\quad\mbox{and}\\
&\frac{1}{N}\left|N-\left(1-\frac{\vep}{N}\right)z(N-2)\right|=\left|1-\left(1-\frac{\vep}{N}\right)\left(1-\frac{2}{N}\right)z\right|\quad \mbox{for }z\in C(0,1),
\end{align*}
and using (\ref{Mobm-bdd}), 
we see that
\begin{align}
\begin{split}\notag
&\Bigg|\frac{1}{2\pi \sqrt{-1}}\int_{\Gamma_2}\frac{1}{z^{n+1}}\times\frac{2z}{[N-N(1-\frac{\vep}{N})z][N-(1-\frac{\vep}{N})z(N-2)]}\notag\\
&\hspace{5.5cm}\times \frac{1}{\frac{1}{N}
\tr\left(\frac{1-(1-\frac{\vep}{N})z}{N-(1-\frac{\vep}{N})z(N-2+2Q)}\right)}dz\Bigg|\notag\\
\end{split}\\
&\hspace{.2cm}\leq \frac{1}{2\pi}\int_{\{\theta\in \R:1/N\leq |\theta|\leq \pi\}}\frac{2}{N\sqrt{2\big(1-\frac{\vep}{N}\big)(1-\cos \theta)}} \notag\\
&\hspace{.5cm}\times\frac{1}{N\sqrt{2\left(1-\frac{\vep}{N}\right)\left(1-\frac{2}{N}\right)(1-\cos\theta)}}\times \frac{1}{\vep/(4+\vep)\times (1/N)}d\theta\notag\\
&\leq \frac{8^{1/2}}{\pi}\times\frac{4+\vep}{\vep}\times \frac{1}{N}\int_{1/N}^\pi \frac{1}{1-\cos \theta}d\theta\leq \frac{4\sqrt{2}}{\pi}\times \frac{4+\vep}{\vep}
,\label{B1:Gamma2}
\end{align}
where the next to the last inequality follows since $\vep/N,2/N\leq 1/2$ and
\begin{align}\label{ineq:cos}
\sup_{r\in (0,1]}r\int_r^\pi \frac{d\theta}{1-\cos \theta}d\theta=2.
\end{align}
Applying (\ref{Gamma1-1}) and (\ref{B1:Gamma2}) to (\ref{Gamma-goal}), we see that $|\alpha(n,0)|$ can be bounded as in (\ref{alpha0-exp-bdd}). 

The proofs of (\ref{alpha1-exp-bdd}) and (\ref{alpha-exp-bdd}) 
follow the same lines as the proof of (\ref{alpha0-exp-bdd}), except for minor changes and the application of the inequality
\[
\sup_{r\in (0,1]}r^2\int_r^\pi \frac{d\theta}{(1-\cos \theta)^{3/2}}\leq \frac{2}{(1-\cos 1)^{1/2}}
\]
instead of (\ref{ineq:cos}) in proving (\ref{alpha-exp-bdd}). The details are left to the readers.
\end{proof}

\begin{lem}\label{lem:T2}
Fix $\vep\in (0,1]$. Then for every $\lambda\in (\vep,\infty)$ such that $(\lambda-\vep)N>\lambda\vep$, $u\in (0,1)$ and $m\in \Bbb N$,
we have
\begin{align}\label{eq:A0nalytic2}
&\left|\frac{\lambda N^2 }{2u(1-u)(N+\lambda)}\sum_{n=mN+1}^\infty \left(\frac{N}{N+\lambda}\right)^n
\left(\frac{1}{N}S_{\beta_u} L_0^{n+1}(J)-\frac{1}{N}S_{\beta_u} L_0^n(J)\right)\right|\notag\\
&\hspace{.5cm}\leq \left[\Big(1+\frac{\lambda}{N}\Big)\left(1-\frac{\vep}{N}\right)\right]^{-(mN+1)}\times \frac{C_\vep\lambda (N-\vep)}{(\lambda-\vep)N-\lambda\vep},
\end{align}
where the constant $C_\vep$ is defined by (\ref{def:Cvep}).
\end{lem}
\begin{proof} 
First we claim the following inequality: for all $n\in \Bbb Z_+$
\begin{align}
\begin{split}
&\left|\frac{1}{N}S_{\beta_u}L^{n+1}_0(J)-\frac{1}{N}S_{\beta_u}L^n_0(J)\right|\leq \frac{2C_\vep u(1-u) }{N^2}
\left(\frac{N}{N-\vep}\right)^n,\label{Budiff-growth}
\end{split}
\end{align}
where the constant $C_\vep$ is defined by (\ref{def:Cvep}).
By (\ref{eq:L0n}) and (\ref{Sbetau}), we have
\begin{align*}
&\frac{1}{N}S_{\beta_u}L_0^{n+1}(J)-\frac{1}{N}S_{\beta_u}L^n_0(J)
=\sum_{s=0}^\infty [\alpha(n+1,s)-\alpha(n,s)]
\frac{u^2N+(u-u^2)\tr(Q^s)}{N}.
\end{align*}
Note that $\sum_{s=0}^\infty \alpha(n,s)=0$, which can be seen by using the initial condition of $\alpha$ and adding up both sides of (\ref{rec1})--(\ref{rec3}). Hence,
by the foregoing display and then the recursive equations (\ref{rec1})--(\ref{rec3}) for $\alpha$, we get
\begin{align*}
&\frac{1}{N}S_{\beta_u}L_0^{n+1}(J)-\frac{1}{N}S_{\beta_u}L_0^n(J)
=\frac{u(1-u)}{N}\sum_{s=0}^\infty [\alpha(n+1,s)-\alpha(n,s)]\tr(Q^s)\notag\\
=&\frac{u(1-u)}{N}\Bigg\{\left[\frac{2}{N^2}+\frac{2}{N^2}\sum_{s=1}^\infty \alpha(n,s)\tr(Q^s)\right]\tr(I)\\
&\hspace{1cm}+\left[\frac{-2}{N^2}-\frac{2}{N}\alpha(n,1)-\frac{2}{N^2}\sum_{s=1}^\infty \alpha(n,s)\tr(Q^s)\right]\tr(Q)\notag\\
&\hspace{1.5cm}-\frac{2}{N}\sum_{s=2}^\infty [\alpha(n,s)-\alpha(n,s-1)]\tr(Q^s)\Bigg\}\\
=&\frac{u(1-u)}{N}\Bigg\{\left[\frac{2}{N}-\frac{2\tr(Q)}{N^2}\right]+\left[\frac{2\tr(Q^2)}{N}-\frac{2\tr(Q)^2}{N^2}\right]\alpha(n,1)\\
&\hspace{1.5cm}-\frac{2\tr(Q)}{N^2}\tr\left(\sum_{s=2}^\infty Q^s\alpha(n,s)\right)+\frac{2}{N}\tr\left(Q\sum_{s=2}^\infty Q^{s}\alpha(n,s)\right)\Bigg\}.
\end{align*}
Our claim (\ref{Budiff-growth}) follows upon applying Lemma~\ref{lem:alpha-growth} to the right-hand of the foregoing equality.

To finish the proof of (\ref{eq:A0nalytic2}), we use (\ref{Budiff-growth}) and get
\begin{align*}
&\Bigg|\frac{\lambda N^2 }{2u(1-u)(N+\lambda)}\sum_{n=mN+1}^\infty \left(\frac{N}{N+\lambda}\right)^n
\left(\frac{1}{N}S_{\beta_u}L^{n+1}_0(J)-\frac{1}{N}S_{\beta_u}L^{n}_0(J)\right)\Bigg|\\
\leq &\frac{\lambda N^2}{2u(1-u)(N+\lambda)}\times \frac{2C_\vep u(1-u) }{N^2}\times \sum_{n=mN+1}^\infty \left(\frac{N}{N+\lambda}\right)^n \left(\frac{N}{N-\vep}\right)^n\\
=& \left[\Big(1+\frac{\lambda}{N}\Big)\left(1-\frac{\vep}{N}\right)\right]^{-(mN+1)}\times \frac{C_\vep\lambda (N-\vep)}{(\lambda-\vep)N-\lambda\vep},
\end{align*}
as stated in in the required inequality (\ref{eq:A0nalytic2}). The proof is complete.
\end{proof}

The final stage is on bounding the third object described below (\ref{T}):
\begin{align}
\begin{split}\label{T'}
&\frac{\lambda N^2}{2u(1-u)(N+\lambda)}
\sum_{n=0}^{mN} \left(\frac{N}{N+\lambda}\right)^n\Bigg[\left(\frac{1}{N}S_{\beta_u}L^{n+1}(J)-\frac{1}{N}S_{\beta_u}L^n(J)\right)\\
&\hspace{4.5cm}-\left(\frac{1}{N}S_{\beta_u}L_0^{n+1}(J)-\frac{1}{N}S_{\beta_u}L_0^n(J)\right)\Bigg].
\end{split}
\end{align}

\begin{lem}\label{lem:LL0-diff}
For all $n\in \Bbb N$, 
\begin{align}\label{eq:LLn}
L^{n}(J)-L^{n}_0(J)=\sum_{j=0}^{n-1}L^j(\eta_{n-j}),
\end{align}
where
$\eta_n\triangleq L\big(L_0^{n-1}(J)\big)-L_0^n(J)$
satisfies
\begin{align}\label{eq:etan}
\begin{split}
\eta_n
=&\sum_{s=0}^\infty \Bigg\{\left[\frac{1}{N}\diag(Q^s)+\diag\big(Q\diag(Q^s)J\big)-2\frac{\tr(Q^s)}{N^2}\right]I\\
&\hspace{.5cm}+\left[-\frac{\diag(Q^s)}{N}Q-Q\frac{\diag(Q^s)}{N}+\frac{2\tr(Q^s)}{N^2}Q\right]\Bigg\}\alpha(n-1,s).
\end{split}
\end{align}
\end{lem}
\begin{proof}
Set $\vep_n=L^n(J)-L^n_0(J)$. For any $n\in \Bbb N$, we have
\begin{align*}
L^{n}(J)=&L\big(L^{n-1}(J)\big)=L(\vep_{n-1})+L\big(L_0^{n-1}(J)\big)
=L(\vep_{n-1})+\eta_{n}+ L_0^{n}(J)
\end{align*}
by the definition of $\eta_{n}$, and 
it follows that 
\[
\vep_{n}=L(\vep_{n-1})+\eta_{n}. 
\]
We obtain (\ref{eq:LLn}) by iterating the above equality and using the equality $\vep_1=\eta_1$ (recall Corollary~\ref{cor:tst} and (\ref{def:L0})).

Next, we show the explicit form (\ref{eq:etan}) of $\eta_n$.
Fix $n\in \Bbb N$, and recall the definition (\ref{def:L0}) of $L_0$.
By (\ref{eq:L0n}) and Corollary~\ref{cor:tst}, we have
\begin{align*}
&L\big(L_0^{n-1}(J)\big)
=\left[J+\frac{2}{N^2}(I-Q)\right]+\sum_{s=0}^\infty \alpha(n-1,s)L(Q^s)\\
=&J+\frac{2}{N^2}(I-Q)+\sum_{s=0}^\infty \alpha(n-1,s)\Bigg\{\frac{N-2}{N}Q^s+\frac{2}{N}Q^{s+1}\\
&-\frac{1}{N}[\diag(Q^s)Q+Q\diag(Q^s)]+\frac{1}{N}\diag(Q^s)+\diag\big(Q\diag(Q^s)J\big)\Bigg\}\\
=&J+\Bigg[\frac{2}{N^2}+\frac{N-2}{N}\alpha(n-1,0)+\sum_{s=0}^\infty \alpha(n-1,s)\frac{1}{N}\diag(Q^s)\\
&\hspace{.5cm}+\sum_{s=0}^\infty \alpha(n-1,s)\diag\big(Q\diag(Q^s)J\big)\Bigg] I\\
&+\Bigg\{\left[-\frac{2}{N^2}+\frac{N-2}{N}\alpha(n-1,1)+\frac{2}{N}\alpha(n-1,0)-\sum_{s=0}^\infty \alpha(n-1,s)\frac{1}{N}\diag(Q^s)\right]Q\\
&\hspace{.5cm}-Q\left[\sum_{s=0}^\infty\alpha(n-1,s)\frac{1}{N}\diag(Q^s)\right]\Bigg\}\\
&+\sum_{s=2}^\infty \left[\frac{N-2}{N}\alpha(n-1,s)+\frac{2}{N}\alpha(n-1,s-1)\right]Q^s.
\end{align*}
On the other hand, if we express the coefficients $\alpha(n,\,\cdot\,)$ of $L^n_0(J)$ in (\ref{eq:L0n}) by $\alpha(n-1,\,\cdot\,)$ using the partial recurrence equations (\ref{rec1})--(\ref{rec3}) obeyed by $\alpha$, then
\begin{align*}
L^n_0(J)
=&J+\left[\frac{2}{N^2}+\alpha(n-1,0)+\frac{2}{N^2}\sum_{s=1}^\infty \alpha(n-1,s)\tr(Q^s)\right]I\\
&+\left[-\frac{2}{N^2}+\frac{N-2}{N}\alpha(n-1,1)-\frac{2}{N^2}\sum_{s=1}^\infty \alpha(n-1,s)\tr(Q^s)\right]Q\\
&+\sum_{s=2}^\infty \left[\frac{N-2}{N}\alpha(n-1,s)+\frac{2}{N}\alpha(n-1,s-1)\right]Q^s\\
=&J+\left[\frac{2}{N^2}+\frac{N-2}{N}\alpha(n-1,0)+\sum_{s=0}^\infty \alpha(n-1,s)\frac{2\tr(Q^s)}{N^2}\right]I\\
&+\left[-\frac{2}{N^2}+\frac{N-2}{N}\alpha(n-1,1)+\frac{2}{N}\alpha(n-1,0)-\sum_{s=0}^\infty \alpha(n-1,s)\frac{2\tr(Q^s)}{N^2}\right]Q\\
&+\sum_{s=2}^\infty \left[\frac{N-2}{N}\alpha(n-1,s)+\frac{2}{N}\alpha(n-1,s-1)\right]Q^s.
\end{align*}
Comparing the last equalities in the foregoing two displays, we deduce that $\eta_n$ satisfies the required expression (\ref{eq:etan}). The proof is complete.
\end{proof}

The next lemma is on bounding the square matrices followed by $\alpha(n-1,s)$ in the expression (\ref{eq:etan}). Recall the definition (\ref{def:RQ}) of $\mathcal R_Q^\gamma(x,s)$ and
the $\ell_1$-norm on $\mathsf M_E$ defined by (\ref{eq:matrix-norm}).

\begin{lem}\label{lem:eigenbdd}
For all $s\in \Bbb Z_+$ and $\gamma\in [0,1]$,
\begin{align}
\begin{split}
& \left\|\diag\big(Q\diag(Q^s)J\big)-\frac{\tr(Q^s)}{N^2}I\right\|\leq \left\|\frac{1}{N}\diag(Q^s)-\frac{\tr(Q^s)}{N^2}I\right\|\\
&\hspace{1.5cm} \leq \min\left\{\frac{4\min_{x\in E}[N-\#\mathcal R_Q^\gamma(x,s)]}{N}+\gamma,\frac{2\tr(|Q|^s;Q<1)}{N}\right\}
 \label{delta-bdd}
 \end{split}
\end{align}
(for the last trace term, recall the notation (\ref{eq:trace-set})).
\end{lem}
\begin{proof}
Fix $s\in \Bbb Z_+$.
We will prove (\ref{delta-bdd}) in four steps.\\

\noindent{\bf (Step 1).} Note that
\begin{align*}
\left\|\diag\big(Q\diag(Q^s)J\big)-\frac{\tr(Q^s)}{N^2}I\right\|&=\frac{1}{N}\sum_{y\in E}\left|\sum_{x\in E} \la y|Q|x\ra \la x|Q^s|x\ra -\frac{\tr(Q^s)}{N}\right|\\
&=\frac{1}{N}\sum_{y\in E}\left|\sum_{x\in E} \la x|Q|y\ra \left[\la x|Q^s|x\ra -\frac{\tr(Q^s)}{N}\right]\right|\\
&\leq \sum_{x\in E} \left|\frac{1}{N}\la x|Q^s|x\ra -\frac{\tr(Q^s)}{N^2}\right|\\
&=\left\|\frac{1}{N}\diag(Q^s)-\frac{\tr(Q^s)}{N^2}I\right\|.
\end{align*}
The first inequality in (\ref{delta-bdd}) follows from the foregoing display.\\

\noindent {\bf (Step 2).} Before giving the proof of the second inequality in (\ref{delta-bdd}), we claim that
for any $x\in E$,
\begin{align}
\left|\langle x|Q^s|x\rangle-\frac{\tr(Q^s)}{N}\right|\leq \frac{2[N-\#\mathcal R_Q^\gamma(x,s)]}{N}+\gamma.\label{eq:locbdd-1}
\end{align}
To see this, consider
\begin{align*}
\left|\langle x|Q^s|x\rangle-\frac{\tr(Q^s)}{N}\right|&\leq \frac{1}{\#\mathcal R_Q^\gamma(x,s)}\sum_{y\in \mathcal R_Q^\gamma(x,s)}\left|\langle x|Q^s|x\rangle-\frac{\#\mathcal R_Q^\gamma(x,s)}{N}\langle y|Q^s|y\rangle\right|\\
&\hspace{.5cm}+\frac{N-\#\mathcal R_Q^\gamma(x,s)}{N}\\
&\leq  \frac{N-\#\mathcal R_Q^\gamma(x,s)}{N}\left[\langle x|Q^s|x\rangle+1\right]+\gamma\\
&\leq \frac{2[N-\#\mathcal R^\gamma_Q(x,s)]}{N}+\gamma,
\end{align*}
which gives our claim (\ref{eq:locbdd-1}).\\

\noindent {\bf (Step 3).} We claim that
\begin{align}
\left\|\frac{1}{N}\diag(Q^s)-\frac{\tr(Q^s)}{N^2}I\right\|\leq \frac{4\min_{x\in E}[N-\#\mathcal R_Q^\gamma(x,s)]}{N}+\gamma.\label{delta-bdd-1-1}
\end{align}
For any $x\in E$,
\begin{align*}
\left\|\frac{1}{N}\diag(Q^s)-\frac{\tr(Q^s)}{N^2}I\right\|=&\frac{1}{N}\sum_{y\in E} \left|\langle y|Q^s|y\rangle-\frac{\tr(Q^s)}{N}\right|\\
\leq &\frac{1}{ N}\sum_{y\in \mathcal R_Q^\gamma(x,s)}\left|\langle y|Q^s|y\rangle-\frac{\tr(Q^s)}{N}\right|+\frac{2[N-\#\mathcal R_Q^\gamma(x,s)]}{N}\\
\leq &\frac{\# \mathcal R_Q^\gamma(x,s)}{N}\left(\frac{2[N-\#\mathcal R^Q_\gamma(x,s)]}{N}+\gamma\right)+\frac{2[N-\#\mathcal R_Q^\gamma(x,s)]}{N}\\
\leq &\frac{4[N-\#\mathcal R_Q^\gamma(x,s)]}{N}+\gamma,
\end{align*}
where the second inequality follows from the inequality (\ref{eq:locbdd-1}). 
Since the last inequality holds for arbitrary $x\in E$, (\ref{delta-bdd-1-1}) follows.\\

\noindent {\bf (Step 4).} Finally, we claim that
\begin{align}
\left\|\frac{1}{N}\diag(Q^s)-\frac{\tr(Q^s)}{N^2}I\right\|\leq \frac{2}{N}\tr\left(|Q|^s;Q<1\right).\label{delta-bdd-1-2}
\end{align}
Let $\{\psi_q;q\in \sigma(Q)\} $ be a basis of $\Bbb C^E$ corresponding to the eigenvalues $q$ of $Q$ ($\sigma(Q)$ denotes the spectrum of $Q$). We may assume that the eigenfunctions $\psi_q$ are orthonormal with respect to the counting measure over $E$. 
Then we have
\begin{align*}
\left\|\frac{1}{N}\diag(Q^s)-\frac{\tr(Q^s)}{N^2}I\right\|&=\frac{1}{N^2}\sum_{x\in E}\left|\sum_{y\in E}\Big(\langle x|Q^s|x\rangle-\langle y|Q^s|y\rangle\Big)\right|\\
&=\frac{1}{N^2}\sum_{x\in E}\left|\sum_{y\in E}\sum_{q\in \sigma(Q)}\Big(\psi_q(x)^2-\psi_q(y)^2\Big) q^s\right|\\
&\leq \frac{2\tr(|Q|^s;Q<1)}{N},
\end{align*}
where the last equality follows since the unique eigenfunction corresponding to $1$ is a constant function. The last inequality gives (\ref{delta-bdd-1-2}).

By (\ref{delta-bdd-1-1}) and (\ref{delta-bdd-1-2}), the second inequality in (\ref{delta-bdd}) follows at once. The proof is complete.
\end{proof}

\begin{lem}\label{lem:T3}
For any $\lambda\in (0,\infty)$, $u\in (0,1)$, $m\in \Bbb N$ and $\gamma\in [0,1]$, we have
\begin{align}
\begin{split}\label{diff0}
&\Bigg|\frac{\lambda N^2}{2u(1-u)(N+\lambda)}
\sum_{n=0}^{mN} \left(\frac{N}{N+\lambda}\right)^n\Bigg[\left(\frac{1}{N}S_{\beta_u}L^{n+1}(J)-\frac{1}{N}S_{\beta_u}L^n(J)\right)\\
&\hspace{1cm}-\left(\frac{1}{N}S_{\beta_u}L_0^{n+1}(J)-\frac{1}{N}S_{\beta_u}L_0^n(J)\right)\Bigg]\Bigg|\leq  \frac{40\Delta^\gamma_Q}{(1-u)}\left(1+\frac{6}{N}\right)^{mN}
\end{split}
\end{align}
where $\Delta_Q^{\gamma}$ is defined in (\ref{def:deltaQ}).
\end{lem}
\begin{proof}
We study the summands of the term in (\ref{T'}) for $1\leq n\leq mN$ (the summand with $n=0$ is zero by Corollary~\ref{cor:tst} and (\ref{def:L0})).
By (\ref{eq:LLn}), we have
\begin{align}
&\left(\frac{1}{N}S_{\beta_u}L^{n+1}(J)-\frac{1}{N}S_{\beta_u}L^n(J)\right)-\left(\frac{1}{N}S_{\beta_u}L_0^{n+1}(J)-\frac{1}{N}S_{\beta_u}L_0^n(J)\right)
\notag\\
=&\frac{S_{\beta_u}}{N}\left(\sum_{j=0}^nL^j\big(\eta_{n+1-j}\big)-\sum_{j=0}^{n-1}L^j\big(\eta_{n-j}\big)\right)\notag\\
\begin{split}\label{p1L0}
=&\frac{S_{\beta_u}}{N}
\left(\eta_{n+1}+\sum_{j=0}^{n-1}(L-I)L^j\big(\eta_{n-j}\big)\right).
\end{split}
\end{align}
Since
$\vertiii{S_{\beta_u}}=u$, (\ref{p1L0}) implies 
\begin{align}
&\left|\left(\frac{1}{N}S_{\beta_u}L^{n+1}(J)-\frac{1}{N}S_{\beta_u}L^n(J)\right)-\left(\frac{1}{N}S_{\beta_u}L_0^{n+1}(J)-\frac{1}{N}S_{\beta_u}L_0^n(J)\right)\right|\notag\\
\leq &\frac{u}{N}\Big(\|\eta_{n+1}\|+\vertiii{L-I}\times\|\eta_n\|+\cdots +\vertiii{L-I}\times \vertiii{L}^{n-1}\times \|\eta_1\|\Big)\notag\\
\leq &\frac{u}{N}\left(\|\eta_{n+1}\|+\frac{4}{N}\|\eta_n\|+\cdots +\frac{4}{N}\|\eta_1\|\right),
\label{eq:LL0-diff01}
\end{align}
where in the last equality we use $\vertiii{L}=1$ and $\vertiii{L-I}\leq 4/N$ by Lemma~\ref{prop:norm}.

To handle the right-hand side of (\ref{eq:LL0-diff01}), we consider the $\ell_1$-norms (defined as in (\ref{eq:matrix-norm})) of the matrices followed by $\alpha(n-1,s)$ in the expression (\ref{eq:etan}) for $\eta_n$. We have
\begin{align}
&\Bigg\|\left[\frac{1}{N}\diag(Q^s)+\diag\big(Q\diag(Q^s)J\big)-2\frac{\tr(Q^s)}{N^2}\right]I\notag\\
&+\left[-\frac{1}{N}\diag(Q^s)Q-Q\frac{1}{N}\diag(Q^s)+\frac{2\tr(Q^s)}{N^2}Q\right]\Bigg\|\notag\\
\leq &\left\|\frac{1}{N}\diag(Q^s)-\frac{\tr(Q^s)}{N^2}I\right\|+\left\|\diag\big(Q\diag(Q^s)J\big)-\frac{\tr(Q^s)}{N^2}I\right\|\notag\\
&+\left\|\left(-\frac{1}{N}\diag(Q^s)+\frac{\tr(Q^s)}{N^2}I\right)Q\right\|\notag\\
&+\left\|Q\left(-\frac{1}{N}\diag(Q^s)+\frac{\tr(Q^s)}{N^2}I\right)\right\|\notag\\
\leq &\left\|\frac{1}{N}\diag(Q^s)-\frac{\tr(Q^s)}{N^2}I\right\|+\left\|\diag\big(Q\diag(Q^s)J\big)-\frac{\tr(Q^s)}{N^2}I\right\|\notag\\
&+\left\|\frac{1}{N}\diag(Q^s)-\frac{\tr(Q^s)}{N^2}I\right\|+\left\|\frac{1}{N}\diag(Q^s)-\frac{\tr(Q^s)}{N^2}I\right\|,\label{eq:LL0-diff2}
\end{align}
where the next to the last equality follows from the fact that $Q$ is a symmetric probability matrix.
Then apply Lemma~\ref{lem:eigenbdd} to (\ref{eq:LL0-diff2}), and we obtain
\begin{align*}
&\Bigg\|\left[\frac{1}{N}\diag(Q^s)+\diag\big(Q\diag(Q^s)J\big)-2\frac{\tr(Q^s)}{N^2}\right]I\notag\\
&\hspace{1cm}+\left[-\frac{1}{N}\diag(Q^s)Q-\frac{1}{N}Q\diag(Q^s)+\frac{2\tr(Q^s)}{N^2}Q\right]\Bigg\|
\leq 16\Delta_Q^{\gamma}
\end{align*}
for $\Delta_Q^{\gamma}$ defined by (\ref{def:deltaQ}). 
Applying (\ref{eq:LL0-diff2}) to the right-hand side of (\ref{eq:etan}), we find that for every $n\in \Bbb N$,
\begin{align*}
\|\eta_n\|\leq &16\Delta_Q^{\gamma}\times \sum_{s=0}^\infty |\alpha(n-1,s)|
\leq   \frac{48\Delta_Q^{\gamma}}{N}\left(1+\frac{6}{N}\right)^{n-1},
\end{align*}
where the last equality follows from Lemma~\ref{lem:bdd-0}.
By the foregoing inequality and (\ref{eq:LL0-diff01}), we obtain that
for all $1\leq n\leq mN$,
\begin{align*}
&\left|\left(\frac{1}{N}S_{\beta_u}L^{n+1}(J)-\frac{1}{N}S_{\beta_u}L^n(J)\right)-\left(\frac{1}{N}S_{\beta_u}L_0^{n+1}(J)-\frac{1}{N}S_{\beta_u}L_0^n(J)\right)\right|\notag\\
&\hspace{.5cm}\leq \frac{48 \Delta_Q^{\gamma}u}{N^2}\left\{\left(1+\frac{6}{N}\right)^{n}+\frac{4}{N}\times \left(1+\frac{6}{N}\right)^{n-1}+\cdots+\frac{4}{N}\times\left(1+\frac{6}{N}\right)+\frac{4}{N}
\right\}\\
&\hspace{.5cm}\leq \frac{80\Delta_Q^{\gamma}u}{N^2}\left(1+\frac{6}{N}\right)^n.
\end{align*}
The foregoing inequality is enough to obtain the required inequality (\ref{diff0}) (recall that the summand with $n=0$ is zero).
\end{proof}

\paragraph{\bf Conclusion for the proof of Theorem~\ref{thmm:main-1}.} The required inequality (\ref{crit:main0}) of the theorem follows plainly by applying
Lemma~\ref{lem:T1}, Lemma~\ref{lem:T2} and Lemma~\ref{lem:T3}, if we recall Lemma~\ref{lem:MUV-exp} (ii) and (iii) and take $u$ to be $1/2$.
\qed

\subsection{Application to large random regular graphs}\label{sec:cor-main2}
We give the proof of Corollary~\ref{cor:main2} in this section, and fix $k\geq 2$. 
Assume that the state space on which $Q^{(n)}$ lives is given by $E_n$ and has size $N_n\nearrow \infty$. 
It follows from a standard result of random regular graphs that for every $s\in \Bbb Z_+$, 
 \[
\lim_{n\to\infty}\frac{\min_{x\in E_n}[N_n-\#\mathcal R^0_{Q^{(n)}}(x,s)]}{N_n}=0\quad\mbox{a.s.}
\]
(cf. Wormald~\cite[Section 2]{W:MRG} for the above convergence as well as the present assumption that the random regular graphs are simple and connected).
We also know that a.s., 
the sequence of empirical eigenvalues distributions of $Q^{(n)}$, namely
\[
\frac{1}{N_n}\#\{r\in (-\infty,q];r\mbox{ is an eigenvalue of }Q^{(n)}\},\quad q\in \R,
\]
converges weakly to the (normalized) Kesten-McKay distribution with density given by (\ref{KM:SM}),
and hence, (\ref{crit:main}) holds. 
See Kesten~\cite{K:SRG} for the fact that the function in (\ref{KM:SM}) is a density of the spectral measure of the random walk kernel on infinite $k$-regular tree, and McKay~\cite[Theorem~4.3]{M:EED} for the above convergence of empirical eigenvalue distributions.

To use (\ref{conv:main}) and find the limit of the ratios of Green functions there,
we apply the spectral representation (\ref{ratio}) and deduce that
\begin{align*}
&\lim_{n\to\infty}\left.\E^{(n)}\left[\int_0^\infty e^{-t\lambda /2}\1_{\{V\}}(X_t^{U})dt\right]\right/
\E^{(n)}\left[\int_0^\infty e^{-t\lambda /2}\1_{\{V\}}(X_t^{V})dt\right]\\
=&\left.\E^{(\infty)}\left[\int_0^\infty e^{-t\lambda /2}\1_{\{y\}}(X_t^{x})dt\right]\right/
\E^{(\infty)}\left[\int_0^\infty e^{-t\lambda /2}\1_{\{y\}}(X_t^{y})dt\right]\\
=&\,\E^{(\infty)}\big[e^{-\lambda H_{x,y}/2}\big],
\end{align*}
almost surely with respect to the randomness that the graphs are chosen,
where $x$ and $y$ are any adjacent vertices on the infinite $k$-regular tree. This proves Corollary~\ref{cor:main2}.

\section{Meeting times of higher orders}\label{sec:MT}
Let $\{U_n;n\in \Bbb Z_+\}$ and $\{V_n;n\in \Bbb Z_+\}$ be two sequences of $Q$-chains with $U_0=V_0$ so that $U_0$ has uniform distribution on $E$, and conditioned on $U_0$, the two chains are independent. We assume in addition that $\{U_n\}$ and $\{V_n\}$ are independent of a system of coalescing $Q$-chains which we may denote it by $\{X^x\}$. Below we consider meeting times of all orders from a slightly more general point of view, using the pairs $(U_m,V_n)$ as starting points.

\begin{prop}\label{prop:MT}
For all $\ell,m,n\in \Bbb Z_+$, 
\begin{align}\label{idd-M}
M_{V_\ell,V_{\ell+m+n}} \stackrel{(\rm d)}{=} M_{U_m,V_n}
\end{align}
and
\begin{align}
\begin{split}\label{eq:shifttime0}
&\int_0^t 2e^{-2(t-v)}\P(M_{V_0,V_{m+n+1}}>v)dv\\
=&\P(M_{U_m,V_n}>t)-e^{-2t}\left(1-\frac{\tr(Q^{m+n})}{N}\right)\\
&+\int_0^t 2e^{-2(t-v)}\frac{1}{N}\sum_{x,y\in E}\langle x|Q^{m+n}|x\rangle\langle x|Q|y\rangle\P(M_{x,y}>v)dv.
\end{split}
\end{align}
\end{prop}

Recall the definition (\ref{def:RQ}) of the sets $\mathcal R^\gamma_Q(x,s)$, and note that, for the last integral in (\ref{eq:shifttime0}),
\begin{align*}
&\left|\frac{1}{N}\sum_{x,y\in E}\langle x|Q^{m+n}|x\rangle \langle x|Q|y\rangle \P(M_{x,y}>v)-\frac{\tr(Q^{m+n})}{N}\P(M_{U,V}>v)\right|\\
\leq &\left\|\frac{\diag(Q^{m+n})}{N}-\frac{\tr(Q^{m+n})}{N^2}I\right\|\\
\leq &\min\left\{\frac{4\min_{x\in E}[N-\#\mathcal R_Q^\gamma(x,m+n)]}{N}+\gamma,\frac{2\tr(|Q|^{m+n};Q<1)}{N}\right\}
\end{align*}
by Lemma~\ref{lem:eigenbdd}, where the bound is independent of $v$. 

\begin{proof}[Proof of Proposition~\ref{prop:MT}]
By the mass-transport equation
\begin{align}\label{mass-transport}
\E[f(U_{n'+1},V_{n'})]=\E[f(U_{n'},V_{n'+1})],\quad \forall\; f:E\times E\lra \R,\;n'\in \Bbb Z_+,
\end{align}
which follows from the reversibility of the $Q$-chains $\{U_{n'}\}$ and $\{V_{n'}\}$, it holds that $M_{U_m,V_n}$ has the same distribution as $M_{U_0,V_{m+n}}=M_{V_0,V_{m+n}}$. By this fact and the stationarity of the chain $\{V_{n'}\}$, (\ref{idd-M}) follows. 

To obtain (\ref{eq:shifttime0}), we condition $(X^{U_n},X^{V_m})$ at the first epoch time and obtain
\begin{align}
\P(M_{U_n,V_m}>t)=&e^{-2t}\P(U_n\neq V_m)\notag\\
&+\int_0^t 2e^{-2(t-v)}\P(U_n\neq V_{m},M_{U_n,V_{m+1}}>v)dv.\label{eq:st1}
\end{align}
Here on the right-hand side of (\ref{eq:st1}), the first term and the distribution $\P(U_n\neq V_m,U_n=x,V_{m+1}=y)$ can be determined respectively as $\P(U_n=V_m)=\tr(Q^{m+n})/N$
and
\begin{align*}
&\P(U_n\neq V_m,U_n=x,V_{m+1}=y)\\
&\hspace{2.1cm}=\P(U_n=x,V_{m+1}=y)-\P(U_n=V_m,U_n=x,V_{m+1}=y)\\
&\hspace{2.1cm}=\frac{1}{N}\langle x|Q^{n+m+1}|y\rangle-\frac{1}{N}\langle x|Q^{n+m}|x\rangle \langle x|Q|y\rangle.
\end{align*}
The proof is complete.
\end{proof}

\begin{cor}\label{cor:Mtime1}
If, for some $n\in \Bbb Z_+$, the return probabilities $\langle x|Q^n|x\rangle $ do not depend on $x$, then
\begin{align*}
\begin{split}
&\int_0^t 2e^{-2(t-v)}\P(M_{V_0,V_{n+1}}>v)dv\\
=&\P(M_{V_0,V_n}>t)-e^{-2t}\left(1-\frac{\tr(Q^{n})}{N}\right)
+\frac{\tr(Q^{n})}{N}\int_0^t 2e^{-2(t-v)}\P(M_{U,V}>v)dv.
\end{split}
\end{align*}
\end{cor}

Hence, in the case that $Q$ is walk regular, the distributions of $M_{U,V_n}$ for all $n\in \Bbb N$ can be completely determined by $M_{U,V}$ by iteration and differential calculus.

\end{document}